\documentclass[12pt]{article}
\usepackage{amsfonts}

\usepackage{tikz,float}
\usepackage[latin1]{inputenc}
\usepackage{amsmath,amssymb}
\allowdisplaybreaks[4]
\usepackage{latexsym}
\usepackage[active]{srcltx}
\usepackage[
bookmarks=true,         
bookmarksnumbered=true, 
colorlinks=true, pdfstartview=FitV, linkcolor=blue, citecolor=blue,
urlcolor=blue]{hyperref}

\newtheorem{theorem}{Theorem}[section]
\newtheorem{corollary}[theorem]{Corollary}
\newtheorem{lemma}[theorem]{Lemma}
\newtheorem{proposition}[theorem]{Proposition}

\newtheorem{remark}[theorem]{Remark}
\numberwithin{equation}{section}
\numberwithin{figure}{section}
\def\C{{\mathbb C}}
\def\R{{\mathbb R}}
\def\E{{{\mathbb E}\,}}

\def\Z{{\mathbb Z}}

\def\G{{\cal G}}
\def\N{{\mathbb N}}

\def\sgn{{\mathop {{\rm sgn\, }}}}
\def\Var{{\mathop {{\rm Var\, }}}}
\def\Cov{{\mathop {{\rm Cov\, }}}}

\def\square{{\vcenter{\vbox{\hrule height.3pt
				\hbox{\vrule width.3pt height5pt \kern5pt
					\vrule width.3pt}
				\hrule height.3pt}}}}

\def\tlint{{- \kern-0.85em \int \kern-0.2em}}
\def\dlint{{- \kern-1.05em \int \kern-0.4em}}

\def\sS {{\cal S}}

\def\ga{{\gamma}}

\newenvironment{proof}[1][Proof]{\noindent\textit{#1.} }{\hfill \rule{0.5em}{0.5em}}

\begin{document}
	\title{Fractional derivatives of local times for some Gaussian processes}	
	\date{\today}
	\author{Minhao Hong, Qian Yu}
	\maketitle
	\begin{abstract}
		In this article, we consider fractional derivatives of local time for $d-$dimensional centered Gaussian processes satisfying certain strong local non-determinism property. We first give a condition for existence of fractional derivatives of the local time defined by Marchaud derivatives in $L^p(p\ge1)$ and show that these derivatives are H\"older continuous with respect to both time and space variables and are also continuous with respect to the order of derivatives. Moreover, under some additional assumptions, we show that this condition is also necessary for existence of derivatives of the local time with the help of contour integration.
	\end{abstract}
\section{Introduction}
Consider a $d$-dimensional centered Gaussian process  $$X=(X^1,\cdots,X^d)=\{X_t=(X^1_t,\dots, X^d_t),\, t\geq 0\}$$ whose components are independent copies of some one-dimensional centered Gaussian process. We assume that $X$ satisfies  the following strong local nondeterminism property {\bf (SLND)} with index $H\in (0,1)$: there exists a positive constant $\kappa_H$ such that for any integer $m\ge 1$, any times $0=s_0<s_1,\dots,s_m,t<\infty$,  
\begin{equation}
	\Var\Big(X_t^1| X^1_{s_1},\dots, X^1_{s_m}\Big)\geq\kappa_H \min\{|t-s_1|^{2H},\cdots,|t-s_m|^{2H},t^{2H}\}. \label{slndp}
\end{equation}
Denote $G^{d,H}_S$ to be the class of all such $d-$dimensional centered Gaussian process $X$ and ${G}^{d,H}_{S,U}$ to be the class of $X\in G^{d,H}_S$ possessing the following property about second moments of increments: There exists a positive constant $K_H$ depending only on $H$ such that 
\begin{equation}
	\E(X^1_t-X^1_s)^2\le K_H(t-s)^{2H} \label{smoi}
\end{equation}
for all $t,s\geq0$. And denote $\widetilde{G}^{d,H}_{S,U}$ to be the class of $X\in {G}^{d,H}_{S,U}$ possessing the following property about covariance of increments: there exists a non-negative decreasing function $\beta_H(\gamma):(1,\infty)\to\mathbb{R}$ with $\lim\limits_{\gamma\to\infty}\beta_H(\gamma)=0$, s.t. for all $0<s<t<\infty$ and $\gamma>1$ with $\frac{t-s}{s}\le\frac{1}{\gamma}$, 
\begin{align}\label{coi}
	\big|\E\big[(X^{1}_{t}-X^{1}_{s})X^{1}_{s}\big]\big|\leq  \beta_H(\gamma)\, \big[\E\big(X^{1}_{t}-X^{1}_{s}\big)^2\big]^{\frac{1}{2}}\big[\E\big(X^{1}_{s}\big)^2 \big]^{\frac{1}{2}}.
\end{align}
Moreover, denote $\widehat{G}^{d,H}_{S,U}$ to be the class of $X\in {G}^{d,H}_{S,U}$ satisfying $\E(X^1_tX^1_s)\ge0$ for all $t,s\ge0$.

 From results in \cite{bgt2004}, \cite{hv2003}, \cite{sxy2019} and Corollary \ref{bfbm}, \ref{sfbm} in the appendix, we can easily see that $d$-dimensional Gaussian processes given below are in $\widetilde{G}^{d,H}_{S,U}\cap\widehat{G}^{d,H}_{S,U}$:
 
 \noindent
 (i) {\it Bi-fractional Brownian motion (bi-fBm)}. The covariance function for components of this process is given by
 \[
 \mathrm{E}(X^{\ell}_t X^{\ell}_s)=2^{-K_0}\big[(t^{2H_0}+s^{2H_0})^{K_0}-|t-s|^{2H_0K_0}\big],
 \]
 where $H_0\in(0,1)$ and $K_0\in(0,1]$.  Here $H=H_0K_0$ and $K_0=1$ gives the classic fractional Brownian motion (fBm) case with Hurst parameter $H=H_0$.
 
 \noindent
 (ii) {\it Sub-fractional Brownian motion (sub-fBm)}.  The covariance function for components of this process  is given by
 \[
 \mathrm{E}(X^{\ell}_t X^{\ell}_s)=t^{2H}+s^{2H}-\frac{1}{2}\big[(t+s)^{2H}+|t-s|^{2H}\big],
 \]
 where $H\in(0,1)$.
 
 For $0<\alpha<1$ and $f:\R\to\R$, define operator $$D^\alpha_{\pm}f(x):=\frac{\alpha}{\Gamma(1-\alpha)}\int_0^{\infty}\frac{f(x)-f(x\mp y)}{y^{\alpha+1}}dy,$$
 where in this article $\Gamma$ is the classic Gamma function. $D^\alpha_{+}$ and $D^{\alpha}_{-}$ are called right-handed and left-handed Marchaud fractional derivatives of order $\alpha$, see \cite{skm1993} for more properties about Marchaud derivatives and \cite{djwz2024,djwx2024} for related SPDEs. Using Lemma 4 in \cite{y1985} gives that if $f$ is $\beta-$H\"older continuous for $\beta\in(0,1]$, which means that for some $K>0$, 
 \[|f(x+y)-f(x)|\le K|y|^{\beta}\text{ for all }x,y\in\R,\]
 then $D^\alpha_{\pm}f(x)$ is $(\beta-\alpha)-$H\"older continuous and $D^\alpha_{\pm}f\in L^2(\R)\cap L^1(\R)$ for any $\alpha\in(0,\beta)$. Using Lemma 5.4 in \cite{b1970} and Remark 1.4 in \cite{hx2020}, we get that for $X\in G^{d,H}_{S}$, the local time $$L(T,x)=\int_{0}^{T}\delta(X_t+x)dt$$ exists in $L^p(p\ge1)$ and has a modification which is $\theta-$H\"older continuous with respect to $x$ for all $\theta\in(0,1\wedge(\frac{1}{H}-d))$, where in this article $\delta$ is the Dirac function and $a\wedge b=\min\{a,b\}$ for all $a,b\in\R$. Then $D^\alpha_{\pm}L(T,\cdot)(x)$ is $(\theta-\alpha)-$H\"older continuous for all $\alpha\in(0,\theta)$ and $\theta\in(0,1\wedge(\frac{1}{H}-d))$, which is called the fractional derivatives of local time for $X$, see e.g. \cite{s2013,y2016,y1985} and references there in. In fact, fractional derivatives of local time has appeared in limit theorems of occupation time of certain self-similar processes such as Bm(\cite{y1986}), stable process(\cite{yg1992}), fBm(\cite{oo2009}), sub-fBm(\cite{s2013}) and recently, some self-similar Gaussian processes(\cite{l2026}). 
 
  When $0<\alpha<1$, using Fourier inverse transform and Lemma \ref{Ff}, which is obtained by contour integration, we get 
 \begin{align*}
 	D^\alpha_{\pm}f(x)=&\frac{\alpha}{\Gamma(1-\alpha)}\int_0^{\infty}\frac{f(x)-f(x\mp y)}{y^{\alpha+1}}dy,\\=&\frac{1}{2\pi}\int_{\R}\Big(\frac{\alpha}{\Gamma(1-\alpha)}\int_0^{\infty}\frac{1}{y^{\alpha+1}}\big(1-e^{\iota \mp y u}\big)dy\Big)\widehat{f}(u)e^{\iota x u}du\\=&\frac{1}{2\pi}\int_{\R}|u|^{{\alpha}}e^{\pm\iota{\alpha}\frac{\pi}{2}\sgn(u)}\widehat{f}(u)e^{\iota xy}du,
 \end{align*}
 where in this article $\iota=\sqrt{-1}$ and $\mathrm{sgn}(u)=\left\{\begin{array}{cl}
 	1,&u>0;\\0,&u=0;\\-1,&u<0. 
 \end{array}\right.$. Now we extend the definition of $D^\alpha_{\pm}f$ to all $\alpha\ge0$ by letting
 \begin{align*}
 	D^\alpha_{\pm}f(x):=\frac{1}{2\pi}\int_{\R}|u|^{{\alpha}}e^{\pm\iota{\alpha}\frac{\pi}{2}\sgn(u)}\widehat{f}(u)e^{\iota xu}du
 \end{align*}
 for any $\alpha\ge0$. Notice that when $\alpha$ is a non-negative integer, using Fourier inverse transform, 
 \begin{align*}
 	\frac{d^\alpha}{dx^\alpha}f(x)&=\frac{d^\alpha}{dx^\alpha}\Big(\frac{1}{2\pi}\int_{\R}\widehat{f}(u)e^{\iota xu}du\Big)\\&=\frac{1}{2\pi}\int_{\R}(\iota u)^{\alpha}\widehat{f}(u)e^{\iota xu}du\\&=\frac{1}{2\pi}\int_{\R}|u|^{{\alpha}}e^{+\iota{\alpha}\frac{\pi}{2}\sgn(u)}\widehat{f}(u)e^{\iota xu}du=D^{\alpha}_{+}f(x),
 \end{align*} 
from which we can get that $D^{\alpha}_{+}f(x)$ is still the $\alpha-$th derivative of $f$ in this case. Moreover, for non-negative integer $\alpha$, 
 \begin{align*}
 	D^{\alpha}_{-}f(x)&=\frac{1}{2\pi}\int_{\R}|u|^{{\alpha}}e^{-\iota{\alpha}\frac{\pi}{2}\sgn(u)}\widehat{f}(u)e^{\iota xu}du\\&=\frac{1}{2\pi}\int_{\R}(-\iota u)^{\alpha}\widehat{f}(u)e^{\iota xu}du\\&=(-1)^\alpha\frac{1}{2\pi}\int_{\R}(\iota u)^{\alpha}\widehat{f}(u)e^{\iota xu}du=(-1)^\alpha D^{\alpha}_{+}f(x),
 \end{align*}
 which means $D^{\alpha}_{+}f(x)$ is the $\alpha-$th derivative of $f$ only when $\alpha$ is even.

   Now for $x=(x_1,\cdots,x_d)$, consider $d-$dimensional heat kernel function 
\[p_{\varepsilon}(x)=\frac{1}{(2\pi\varepsilon)^{\frac{d}{2}}}e^{-\frac{|x|^2}{2\varepsilon}}=\frac{1}{(2\pi)^d}\int_{\R^d}e^{\iota y\cdot x}e^{-\frac{\varepsilon}{2}|y|^2}dy\]
and for multi-index $\boldsymbol{\alpha}=(\alpha_1,\cdots,\alpha_d)$ with non-negative real numbers $\alpha_i$, the right-handed and left-handed  $\boldsymbol{\alpha}-$th partial derivatives of $p_{\varepsilon}(x)$ are given by 
\begin{align}\label{fdalpha}
	D^{(\boldsymbol{\alpha})}_{\pm}p_{\varepsilon}(x)=\frac{1}{(2\pi)^d}\int_{\R^d}\Big(\prod_{\ell=1}^{d}|{y_\ell}|^{{\alpha}_\ell}{e}^{\pm\iota{\frac{\pi{\alpha}_\ell}{2}\mathrm{sgn}(y_\ell)}}\Big)e^{\iota y\cdot x}e^{-\frac{\varepsilon}{2}|y|^2}dy,
\end{align}

Like \cite{ghx2019}, for any $T>0$, $x\in\R^d$ and $X\in G^{d,H}_{S}$, define 
\begin{align}\label{approx}
	L^{(\boldsymbol{\alpha})}_{\pm,\varepsilon}(T,x):=\int_{0}^{T}D^{(\boldsymbol{\alpha})}_{\pm}p_{\varepsilon}(X_t+x)dt.
\end{align}
If $L^{(\boldsymbol{\alpha})}_{+,\varepsilon}(T,x)$ or $L^{(\boldsymbol{\alpha})}_{-,\varepsilon}(T,x)$ converges in $L^p$ when $\varepsilon\downarrow0$, we denote the limit by $L^{(\boldsymbol{\alpha})}_{+}(T,x)$ or $L^{(\boldsymbol{\alpha})}_{-}(T,x)$ and call it the right-handed or  left-handed $\boldsymbol{\alpha}-$th derivative of local time for $X$. If it exists, $L^{(\boldsymbol{\alpha})}_{\pm}(T,x)$ admits the following $L^p-$representation 
\[L^{(\boldsymbol{\alpha})}_{\pm}(T,x):=\int_{0}^{T}D^{(\boldsymbol{\alpha})}_{\pm}\delta(X_t+x)dt.\]
\begin{remark}\label{inca}
	Because the components of $X\in G^{d,H}_{S}$ are all centered Gaussian Processes and the fact that
	\begin{align*}
		L^{(\boldsymbol{\alpha})}_{+,\varepsilon}(T,x)&=\frac{1}{(2\pi)^d}\int_{0}^{T}\int_{\R^d}\Big(\prod_{\ell=1}^{d}|{y_\ell}|^{{\alpha}_\ell}{e}^{+\iota{\frac{\pi{\alpha}_\ell}{2}\mathrm{sgn}(y_\ell)}}\Big)\prod_{\ell=1}^de^{\iota y_\ell(X^\ell_t+x)}e^{-\frac{\varepsilon}{2}|y|^2}dydt\\&=\frac{1}{(2\pi)^d}\int_{0}^{T}\int_{\R^d}\Big(\prod_{\ell=1}^{d}|{y_\ell}|^{{\alpha}_\ell}{e}^{-\iota{\frac{\pi{\alpha}_\ell}{2}\mathrm{sgn}(y_\ell)}}\Big)\prod_{\ell=1}^de^{\iota y_\ell(-X^\ell_t-x)}e^{-\frac{\varepsilon}{2}|y|^2}dydt,
	\end{align*}
$L^{(\boldsymbol{\alpha})}_{+,\varepsilon}(T,x)$ has the same distribution as $L^{(\boldsymbol{\alpha})}_{-,\varepsilon}(T,-x)$. Then we can get  $L^{(\boldsymbol{\alpha})}_{+,\varepsilon}(T,x)$ converges in $L^p$ if and only if $L^{(\boldsymbol{\alpha})}_{-,\varepsilon}(T,-x)$ converges in $L^p$ and at this time $L^{(\boldsymbol{\alpha})}_{+}(T,x)$ and $L^{(\boldsymbol{\alpha})}_{-}(T,-x)$ has the same distribution.
\end{remark}
\begin{remark}
	When $\alpha_1,\cdots,\alpha_d$ are all non-negative integers, $L^{(\boldsymbol{\alpha})}_{+}(T,x)$ becomes the $\boldsymbol{\alpha}$-th derivative of local time $L^{(\boldsymbol{\alpha})}_1(T,x)$ in Remark 1.4 in \cite{hx2020}. In particular, $L^{(\boldsymbol{0})}_1(T,x)$ is just the local time of the $d$-dimensional Gaussian process $X$ at $-x$. When $X^1$ are fBm, bi-fBm or sub-fBm, $L^{(\boldsymbol{\alpha})}_1(T,x)$ exists in $L^2$ if and only if $H(2|\boldsymbol{\alpha}|+d)<1$ from the results of \cite{hx2020,hx2021}, where $|\boldsymbol{\alpha}|=\sum_{\ell=1}^{d}\alpha_\ell$ for non-negative integers $\alpha_1,\cdots,\alpha_d$. For more about local time and its integer-order derivatives for stochastic processes or random fields, we recommend e.g. \cite{no2007,s2011,wx2010,xy2023,y2014,yyc2017,ysz2015,zsy2023} and references there in. 
\end{remark}

   In this paper, we consider the case when one of $\alpha_1,\cdots,\alpha_d$ is not integer, where the existence of $L^{(\boldsymbol{\alpha})}_{\pm}(T,x)$ in $L^p$ still remains to be solved. Given any non-negative real numbers $\alpha_1,\cdots,\alpha_d$, we denote $|\boldsymbol{\alpha}|=\sum_{\ell=1}^{d}\alpha_\ell$ and the following are our main results. First, we give a sufficient condition for the existence of $L^{(\boldsymbol{\alpha})}_{\pm}(T,x)$ in $L^p$ and then show its H\"older continuity w.r.t both $T$ and $x$. Moreover, we will also show that $L^{(\boldsymbol{\alpha})}_{\pm}(T,x)$ is continuous in $L^p$ w.r.t. the order of derivatives $\boldsymbol{\alpha}$.
\begin{theorem}\label{suf}
	Suppose that $X=\{X_t,t\ge0\}$ are Gaussian processes in $G^{d,H}_S$ with $H\in(0,1)$, $x=(x_1,x_2,\cdots,x_d)\in\R^d$ and the components of multi-index $\boldsymbol{\alpha}=(\alpha_1,\cdots,\alpha_d)$ satisfy $\alpha_\ell\ge0$ for all $\ell=1,2,\cdots,d$. Then if  $H(2|\boldsymbol{\alpha}|+d)<1$, $L^{(\boldsymbol{\alpha})}_{+}(T,x)$ and $L^{(\boldsymbol{\alpha})}_{-}(T,x)$ exist in $L^p$ for any $p\in[1,\infty)$. Moreover, the following statements exist: 
	
	{\rm (1).} $L^{(\boldsymbol{\alpha})}_{\pm}(T,x)$ has a continuous modification which is locally $\theta_1-$H\"older continuous w.r.t. time variable $T$ for all $\theta_1\in(0,1-H(|\boldsymbol{\alpha}|+d))$.
	
	{\rm (2).} $L^{(\boldsymbol{\alpha})}_{\pm}(T,x)$ has a continuous modification which is locally $\theta_2-$H\"older continuous w.r.t. space variable $x$ for all $\theta_2\in(0,1\wedge\frac12(\frac{1}{H}-2|\boldsymbol{\alpha}|-d))$;
	
	{\rm (3).}  For $\boldsymbol{\beta}=(\beta_1,\beta_2,\cdots,\beta_d)\in[0,\infty)^d$ with $H(2|\boldsymbol{\beta}|+d)<1$, if $\beta_\ell$ converges to $\alpha_\ell$ for all $\ell=1,2,\cdots,d$, then $L^{(\boldsymbol{\beta})}_{\pm}(T,x)$ converges to $L^{(\boldsymbol{\alpha})}_{\pm}(T,x)$ in $L^p$.
\end{theorem}

Using Lemma \ref{Ff}, Lemma \ref{diff} and Lemma \ref{fouriert}, we will get that with some additional assumptions, $L^{(\boldsymbol{\alpha})}_{\pm,\varepsilon}(T,x)$ will diverges in $L^2$ as $\varepsilon\downarrow0$, which means that $H(2|\boldsymbol{\alpha}|+d)<1$ is the sufficient and necessary condition for the existence of $L^{(\boldsymbol{\alpha})}_{\pm}(T,x)$ in $L^2$ for those processes.

\begin{theorem}\label{nes}
	Use notations in Theorem \ref{suf} and $\Z$ denotes the set of all integer. Assume $X\in G^{d,H}_{S,U}$ and $H(2|\boldsymbol{\alpha}|+d)\ge1$. 
	
	When $x=0$, the following statements exist:
	
	{\rm (1).} If $Hd\le1$, $L^{(\boldsymbol{\alpha})}_{\pm,\varepsilon}(T,0)$ diverges in $L^2$ as $\varepsilon\downarrow0$;
	
	{\rm (2).} If $Hd>1$ and $\sum\limits_{\ell=1,\alpha_\ell\in\Z}^{d}\alpha_\ell$ is even, $L^{(\boldsymbol{\alpha})}_{\pm,\varepsilon}(T,0)$ diverges in $L^2$; 
	
	{\rm (3).} If $Hd>1$ and $\sum\limits_{\ell=1,\alpha_\ell\in\Z}^{d}\alpha_\ell$ is odd, $L^{(\boldsymbol{\alpha})}_{\pm,\varepsilon}(T,0)$ diverges in $L^2$ for $X\in\widehat{G}^{d,H}_{S,U}$.
	
	When $x\ne0$, the following statements exist:
	
	{\rm (4).} $L^{(\boldsymbol{\alpha})}_{\pm,\varepsilon}(T,x)$ diverges in $L^2$ as $\varepsilon\downarrow0$ for $X\in\widetilde{G}^{d,H}_{S,U}$ in the case that there exists an $\ell\in\{1,2,\cdots,d\}$ s.t. $\alpha_\ell$ is integer and $x_\ell\ne0$;
\end{theorem}
\begin{remark}
	If $\alpha_\ell$ is not integer for all $\ell=1,2,\cdots,d$ with $x_\ell\ne0$, the method used in the proof of statement {\rm (4)} of Theorem \ref{nes} doesn't work, which is caused by the fact that a power function with fractional exponent is not analytic so we couldn't get the result in Lemma \ref{fouriert}. However, we think in this case $L^{(\boldsymbol{\alpha})}_{\pm,\varepsilon}(T,x)$ still diverges in $L^2$ as $\varepsilon\downarrow0$ for $X\in\widetilde{G}^{d,H}_{S,U}$, which remains to be proven.
\end{remark}

After some preliminaries in Section 2, Section 3 is devoted to the proof of Theorem \ref{suf} and Section 4 is devoted to the proof of Theorem \ref{nes}. In Section 5 we prove some technical lemmas with the help of contour integration  and in Section 6 we give the proof of {\bf SLND} \eqref{slndp} for bi-fBm and sfBm. Throughout this paper, if not mentioned otherwise, the letter $c$, 
with or without a subscript or superscript, denotes a generic positive finite constant whose exact value may change from line to line. We denote $\Z$ the set of all integer and $\C$ the set of all complex numbers. For any $x,y\in\mathbb{R}^d$, we denote $x\cdot y$ the usual inner product and $|x|=(\sum\limits^d_{i=1}|x_i|^2)^{1/2}$. Moreover, we denote 
\begin{align}\label{alphapower}
	(\iota x)_0^{\alpha}&:=|x|^{\alpha}e^{\iota\frac{\pi \alpha}{2}\sgn(x)}
	\end{align}
for any $x\in\R$ and
\begin{align}\label{alphastar}
\alpha^*&:=\widetilde{\alpha}+(\bar{\alpha}\wedge1)\mathbf{1}_{\{\widetilde{\alpha}=0\}}
\end{align} 
for any $\alpha\ge0$, where $\bar{\alpha}$ is the largest integer no greater than $\alpha$, $\widetilde{\alpha}=\alpha-\bar{\alpha}$ and $\mathbf{1}_{A}$ is the indicator function of some set $A$.
\section{Preliminaries}
In this section, we will give some lemmas prepared for the proof of our main result. First consider Gaussian vector $(Y_1, \dots, Y_n)$ and use $\mathrm{Cov}(Y_1, \dots, Y_n)$ to denote its covariance matrix, whose determinant can be evaluated by using the formula \begin{equation}\label{decomCov}
	\det \mathrm{Cov}(Y_1, \dots, Y_n) = \mathrm{Var}(Y_1) \prod_{m=1}^n \mathrm{Var}(Y_m|Y_1, \dots, Y_{m-1}),
\end{equation}
which can be found in Corollary A.2 in \cite{km2020}. Using these notations, we introduce a lemma from \cite{cd1982}, which will play an important role in the proof of Theorem \ref{suf}. 
\begin{lemma}[Lemma 2 of \cite{cd1982}]\label{Lem:cd1982}
	Let $Y_1, \dots, Y_n$ be mean zero Gaussian random variables that are linearly independent and assume that 
	$\int_\R g(v) e^{-\varepsilon v^2} dv < \infty$ for all $\varepsilon > 0$. Then
	\[ \int_{\R^n} g(v_1) \exp\bigg[ -\frac{1}{2} \mathrm{Var}\bigg( \sum_{l=1}^n v_l Y_l \bigg) \bigg] dv_1 \dots dv_n 
	= \frac{(2\pi)^{(n-1)/2}}{\det\mathrm{Cov}(Y_1, \dots, Y_n)^{1/2}}
	\int_\R g(v/\sigma_1) e^{-v^2/2} dv, \]
	where $\sigma_1^2 = \mathrm{Var}(Y_1|Y_2, \dots, Y_n)$.
\end{lemma}
Then we give a lemma about the covariance of increments of $X$ in $\widetilde{G}^{d,H}_{S,U}$, which will be used in the proof of statement {\rm (4)} of Theorem \ref{nes}.
\begin{lemma}\label{bocoi}
	 For $X$ in $\widetilde{G}^{d,H}_{S,U}$, there exists a positive constant $K$ and a non-negative decreasing function $\widetilde{\beta}(\gamma):(1,\infty)\to\mathbb{R}$ with $\lim\limits_{\gamma\to\infty}\widetilde{\beta}(\gamma)=0$, s.t. for all $t>s>0$ and $\gamma>1$, 
	\begin{align}
		\big|\E\big[(X_t^1-X_s^1)X_s^1\big]\big|\le K\big(\widetilde{\beta}(\gamma)s^{H}(t-s)^{H}+\gamma^H(t-s)^{2H}\big).
	\end{align}
\end{lemma}
\begin{proof}
	 When $t>s$, we have 	\begin{align*}
		\big|\E\big[(X^{1}_{t}-X^{1}_{s})X^{1}_{s}\big]\big|&\leq  \beta_H(\gamma) \big[\E\big(X^{1}_{t}-X^{1}_{s}\big)^2\big]^{\frac{1}{2}}\big[\E\big(X^{1}_{s}\big)^2 \big]^{\frac{1}{2}}\\&\leq K_{H}\beta_H(\gamma)s^{H}(t-s)^{H}
	\end{align*}
for all $t>s>0$ and $\gamma>1$ with $\frac{t-s}{s}\le\frac{1}{\gamma}$ by \eqref{smoi} and \eqref{coi}. And we have 
\begin{align*}
	\big|\E\big[(X^{1}_{t}-X^{1}_{s})X^{1}_{s}\big]\big|&\leq  \big[\E\big(X^{1}_{t}-X^{1}_{s}\big)^2\big]^{\frac{1}{2}}\big[\E\big(X^{1}_{s}\big)^2 \big]^{\frac{1}{2}}\\&\leq K_H (t-s)^Hs^H\\&\le K_H \gamma^H(t-s)^{2H}
\end{align*}
for all $t>s>0$ and $\gamma>1$ with $\frac{t-s}{s}>\frac{1}{\gamma}$ by Cauchy-Schwarz inequality and \eqref{smoi}. Combining these cases we get the desired result.
\end{proof}

Then we introduce a lemma which has appeared in \cite{wx2010} and will be used in the proof of Theorem \ref{nes}.
\begin{lemma}[Lemma 2.2 of \cite{wx2010}]\label{lmaxiao}
	Let $\beta$, $\gamma$ and $p$ be
	positive constants, then 
	\begin{equation}
		\int_0^1\frac{r^{p-1}}{\big(A+r^{\ga}\big)^\beta}\, dr
		\asymp\left\{\begin{array}{ll}
			A^{\frac p \ga -\beta} \qquad\qquad \qquad &\hbox{ if } \,\,\beta\gamma> p, \\
			\log\big(1+A^{-1/\ga}\big)  &\hbox{ if }\,\,\beta\gamma=p, \\
			1   &\hbox{ if } \,\, \beta\gamma<p.
		\end{array}
		\right.
	\end{equation}
	for all $A \in (0, 1)$. In the above,  $f(A) \asymp g (A)$  means that the ratio $f(A)/g(A)$
	is bounded from below and above by positive constants that do not
	depend on $A \in (0, 1)$.
\end{lemma}

Moreover, we give a lemma in \cite{hx2020}, which will be used in the proof of Theorem \ref{suf}.
\begin{lemma}[Lemma A.2 in \cite{hx2020}]\label{lmahx2020} For any $T>0$ and $a_i\in(0,1)$ with $i=1,2,\cdots,m$, 
		\begin{align*}
			\int_{D^m_{T,h}}\prod^{m}_{j=1} u_j^{-a_i}\, du\leq \frac{\prod^m_{j=1}\Gamma(1-a_j)}{\Gamma(m+1-\sum^{m}_{i=1}a_i)}h^{\sum\limits^{m}_{i=1}(1-a_i)},
		\end{align*}
		where $D^m_{T,h}=\Big\{T<u_1, T<u_1+u_2+\cdots+u_m<T+h:\, u_i>0,\, i=2,\cdots,m\Big\}$.
\end{lemma}

Finally, we give a lemma in \cite{hx2021}, which will be used in the proof of statement {\rm (2)} and statement {\rm (3)} in Theorem \ref{nes}.
\begin{lemma}[Lemma 2.3 of \cite{hx2021}]\label{lmahx}
	Assume that $k\in\N\cup\{0\}$ and $\varepsilon>0$. Then, for any $a,b,c\in\R$ with $a,c>0$ and $(a+\varepsilon)(c+\varepsilon)-b^2>0$, we have
	\begin{align*}
		&\frac{(-1)^k}{2\pi}\int_{\R^2} \exp\Big\{-\frac{1}{2}(y_2^2a+2y_2 y_1b+y_1^2c)-\frac{\varepsilon}{2}(y_2^2+y_1^2)\Big\}  y^{k}_2 y_1^{k} \, dy\\
		&\qquad\qquad\qquad\qquad\qquad\qquad=\sum^k_{\ell=0, \text{even}}    \frac{c_{k,\ell} \,b^{k-\ell}}{((a+\varepsilon)(c+\varepsilon)-b^2)^{\frac{2k-\ell+1}{2}}},
	\end{align*}
	where $c_{k,\ell}=(\ell-1)!!  \binom{k}{\ell} (2k-\ell-1)!!$.
\end{lemma}
\section{Proof of Theorem \ref{suf}}
In this section, we focus on the proof of Theorem \ref{suf}. 
Before starting our proof, we consider two propositions for $X\in G^{d,H}_{S}$. 
\begin{proposition}\label{upper1}
	For any $\boldsymbol{\alpha}=(\alpha_1,\cdots,\alpha_d)\in[0,\infty)^d$ with $H(2|\boldsymbol{\alpha}|+d)<1$ and $X\in G^{d,H}_{S}$, there exists a constant $c_{\boldsymbol{\alpha}}>0$ s.t. 
	\begin{align*}
		\prod_{\ell=1}^{d}\bigg(\int_{\R^{m}}\prod_{j=1}^{m}|y_{j\ell}|^{\alpha_\ell}&\times\exp\bigg(-\frac12\Var\Big(\sum_{j=1}^{m}y_{j\ell}X^{\ell}_{t_j}\Big)\bigg)dy_{\cdot\ell}\bigg)\\&\le c_{\boldsymbol{\alpha}} \sum_{\sS}\prod_{\ell=1}^{d}\prod_{j=1}^{m}(t_j-t_{j-1})^{-H(p_{j\ell}\alpha_\ell+\bar{p}_{j\ell}\alpha_\ell+1)}
	\end{align*}
	for any $0<t_1<t_2<\cdots<t_n<\infty$, where $dy_{\cdot\ell}=dy_{1\ell}\cdots dy_{m\ell}$ for all $\ell=1,2,\cdots,d$ and 
	$$\sS=\{p_{j\ell},\bar{p}_{j\ell}\in\{0,1\}:\bar{p}_{1\ell}=0,p_{m\ell}=1,p_{j\ell}+\bar{p}_{j\ell}=1,j=1,2,\cdots,m,\ell=1,2,\cdots,d\}.$$
\end{proposition}
\begin{proof}
First, for each $\ell$, using H\"older's inequality and Lemma \ref{Lem:cd1982}, 
\begin{equation}\label{pthm12}
	\begin{split}
		&\int_{\R^{m}}\prod_{j=1}^{m}|y_{j\ell}|^{\alpha_\ell}\times\exp\bigg(-\frac12\Var\Big(\sum_{j=1}^{m}y_{j\ell}X^{\ell}_{t_j}\Big)\bigg)dy_{\cdot\ell}\\\le&\prod_{j=1}^{m}\bigg(\int_{\R^{m}}|y_{j\ell}|^{m\alpha_\ell}\times\exp\bigg(-\frac12\Var\Big(\sum_{j=1}^{m}y_{j\ell}X^{\ell}_{t_j}\Big)\bigg)dy_{\cdot\ell}\bigg)^{\frac1m}\\=&\prod_{j=1}^{m}\bigg(\frac{(2\pi)^{\frac{m-1}{2}}}{\det\Cov(X^1_{t_1},\cdots,X^1_{t_m})^{\frac12}}\int_{\R}\Big(\frac{|y_{j\ell}|}{\sigma_{t_j}}\Big)^{m\alpha_\ell}e^{-\frac{y_{j\ell}^2}{2}}dy_{j\ell}\bigg)^{\frac1m}\\= &(2\pi)^{\frac{m-1}{2m}}\Big(\int_{\R}|v|^{m\alpha_\ell}e^{-\frac{v^2}{2}}dv\Big)\times\Big(\det\Cov(X^1_{t_1},\cdots,X^1_{t_m})\Big)^{-\frac12}\times\prod_{j=1}^{m}\sigma_{t_j}^{-\alpha_\ell},
	\end{split}
\end{equation}
where $\sigma_{t_j}:=\Var(X^1_{t_j}|X^1_{t_i},i=1,\cdots,m,i\ne j)$ for $j=1,2,\cdots,d$. Then using \eqref{decomCov} and {\bf SLND} \eqref{slndp} gives that  
\begin{equation}\label{pthm13}
	\begin{split}
		&\Big(\det\Cov(X^1_{t_1},\cdots,X^1_{t_m})\Big)^{-\frac12}\times\prod_{j=1}^{m}\sigma_{t_j}^{-\alpha_\ell}\\\le&(\kappa_{H})^{-m(\alpha_\ell+\frac12)} \prod_{j=1}^{m}(t_j-t_{j-1})^{-H}\times(t_m-t_{m-1})^{-H\alpha_\ell}\times\prod_{j=1}^{m-1}\Big(\min\big\{t_j-t_{j-1},t_{j+1}-t_j\big\}\Big)^{-H\alpha_\ell}\\\le& (\kappa_{H})^{-m(\alpha_\ell+\frac12)}\prod_{j=1}^{m}(t_j-t_{j-1})^{-H}\times(t_m-t_{m-1})^{-H\alpha_\ell}\times\prod_{j=1}^{m-1}\Big((t_j-t_{j-1})^{-H\alpha_\ell}+(t_{j+1}-t_j)^{-H\alpha_\ell}\Big)\\=&(\kappa_{H})^{-m(\alpha_\ell+\frac12)} \sum_{\sS_\ell}\prod_{j=1}^{m}(t_j-t_{j-1})^{-H(p_{j\ell}\alpha_\ell+\bar{p}_{j\ell}\alpha_\ell+1)},
	\end{split}
\end{equation}
where $0=t_0<t_1<t_2<\cdots<t_m$ and  $$\sS_\ell=\{p_{j\ell},\bar{p}_{j\ell}\in\{0,1\}:\bar{p}_{1\ell}=0,p_{m\ell}=1,p_{j\ell}+\bar{p}_{j\ell}=1,j=1,2,\cdots,m\},\ell=1,2,\cdots,d.$$	
So combining \eqref{pthm12} with \eqref{pthm13} we get the desired result.
\end{proof}

Then we discuss the sufficient condition for the existence of $L^{(\boldsymbol{\alpha})}_{\pm}(T,x)$ in $L^p(p\ge1)$. 
From \eqref{fdalpha} and \eqref{approx}, we have 
\begin{align*}
	L^{(\boldsymbol{\alpha})}_{\pm,\varepsilon}(T,x)-L^{(\boldsymbol{\alpha})}_{\pm,\eta}(T,x)=\frac{1}{(2\pi)^d}\int_{0}^{T}\int_{\R^d}\prod_{\ell=1}^{d}(\pm\iota y_{\ell})_0^{\alpha_\ell}e^{\iota y\cdot (x+X_t)}\Big(e^{-\frac{\varepsilon}{2}|y|^2}-e^{-\frac{\eta}{2}|y|^2}\Big)dydt
\end{align*}
and 
\begin{align}
		&\E\big[L^{(\boldsymbol{\alpha})}_{\pm,\varepsilon}(T,x)-L^{(\boldsymbol{\alpha})}_{\pm,\eta}(T,x)\big]^m\nonumber\\\le&\frac{1}{(2\pi)^{md}}\int_{[0,T]^m}\int_{\R^{md}}\prod_{j=1}^{m}\prod_{\ell=1}^{d}|y_{j\ell}|^{\alpha_\ell}\times\exp\bigg(-\frac12\sum_{\ell=1}^{d}\Var\Big(\sum_{j=1}^{m}y_{j\ell}X^{\ell}_{t_j}\Big)\bigg)dydt\nonumber\\=&\frac{m!}{(2\pi)^{md}}\int_{[0,T]^m_<}\prod_{\ell=1}^{d}\bigg(\int_{\R^{m}}\prod_{j=1}^{m}|y_{j\ell}|^{\alpha_\ell}\times\exp\bigg(-\frac12\Var\Big(\sum_{j=1}^{m}y_{j\ell}X^{\ell}_{t_j}\Big)\bigg)dy_{\cdot\ell}\bigg)dt\label{pthm11}
\end{align} 
for any even $m$ and $\varepsilon,\eta>0$, where $(\iota x)_0^{\alpha}=|x|^{\alpha}e^{\iota\frac{\pi \alpha}{2}\sgn(x)}$ for any $\alpha\ge0$ is defined in \eqref{alphapower},  and $$[0,T]^m_<=\{(t_1,\cdots,t_m):0=t_0<t_1<t_2<\cdots<t_m<T\}.$$ Using Proposition \ref{upper1} gives that 
 \begin{equation}\label{pthm1b}
 	\begin{split}
 		\int_{[0,T]^m_<}\prod_{\ell=1}^{d}\bigg(\int_{\R^{m}}\prod_{j=1}^{m}|y_{j\ell}|^{\alpha_\ell}\times&\exp\bigg(-\frac12\Var\Big(\sum_{j=1}^{m}y_{j\ell}X^{\ell}_{t_j}\Big)\bigg)dy_{\cdot\ell}\bigg)dt\\\le&c\sum_{\sS} \int_{[0,T]^m_<}\prod_{\ell=1}^{d}\prod_{j=1}^{m}(t_j-t_{j-1})^{-H(p_{j\ell}\alpha_\ell+\bar{p}_{j\ell}\alpha_\ell+1)}dt\\\le&c\sum_{\sS} \int_{[0,T]^m}\prod_{j=1}^{m}u_j^{-H\sum_{\ell=1}^{d}(p_{j\ell}\alpha_\ell+\bar{p}_{j\ell}\alpha_\ell+1)}du<\infty.
 	\end{split}
 \end{equation}
 where in the second inequality we use the coordinate transform $u_j=t_j-t_{j-1},j=1,\cdots,m$ and the last inequality is from the fact 
 $$\displaystyle\sum_{\ell=1}^{d}H(p_{j\ell}\alpha_\ell+\bar{p}_{j\ell}\alpha_\ell+1)\le H(2|\boldsymbol{\alpha}|+d)<1,$$
 for any $p_{j\ell}$s and $\bar{p}_{j\ell}$s in $\sS$. It leads to that when applying dominated convergence theorem to \eqref{pthm11}, we can get 
  \[\lim\limits_{\varepsilon,\eta\downarrow0}\E\big[L^{(\boldsymbol{\alpha})}_{\pm,\varepsilon}(T,x)-L^{(\boldsymbol{\alpha})}_{\pm,\eta}(T,x)\big]^m=0,\]
which means that $L^{(\boldsymbol{\alpha})}_{\pm,\varepsilon}(T,x)$ converges in $L^m$ for any even $m\ge2$ so  $L^{(\boldsymbol{\alpha})}_{\pm}(T,x)$ exists in $L^p(p\ge1)$.

Then we turn to the proof of H\"older continuity. For the time variable and even $m$, from \eqref{fdalpha} and \eqref{approx}, 
\begin{align*}
	L^{(\boldsymbol{\alpha})}_{\pm,\varepsilon}(T+h,x)-L^{(\boldsymbol{\alpha})}_{\pm,\varepsilon}(T,x)&=\frac{1}{(2\pi)^d}\int_{T}^{T+h}\int_{\R^d}\prod_{\ell=1}^{d}(\pm\iota y_{\ell})_0^{\alpha_\ell}e^{\iota y\cdot (x+X_t)}e^{-\frac{\varepsilon}{2}|y|^2}dydt.
\end{align*} 
Using the similar arguments in \eqref{pthm11}, 
\begin{equation}
	\begin{split}
		&\E\big[L^{(\boldsymbol{\alpha})}_{\pm,\varepsilon}(T+h,x)-L^{(\boldsymbol{\alpha})}_{\pm,\varepsilon}(T,x)\big]^m\\\le&\frac{m!}{(2\pi)^{md}}\int_{[T,T+h]^m_<}\prod_{\ell=1}^{d}\bigg(\int_{\R^{m}}\prod_{j=1}^{m}|y_{j\ell}|^{\alpha_\ell}\times\exp\bigg(-\frac12\Var\Big(\sum_{j=1}^{m}y_{j\ell}X^{\ell}_{t_j}\Big)\bigg)dy_{\cdot\ell}\bigg)dt,
	\end{split}
\end{equation}
where $[T,T+h]^m_<=\{(t_1,\cdots,t_m):T<t_1<t_2<\cdots<t_m<T+h\}$. Similarly, using Proposition \ref{upper1}, 
\begin{equation}
	\begin{split}
		&\E\big[L^{(\boldsymbol{\alpha})}_{\pm,\varepsilon}(T+h,x)-L^{(\boldsymbol{\alpha})}_{\pm,\varepsilon}(T,x)\big]^m\le c\sum_{\sS} \int_{[T,T+h]^m_<}\prod_{\ell=1}^{d}\prod_{j=1}^{m}(t_j-t_{j-1})^{-H(p_{j\ell}\alpha_\ell+\bar{p}_{j\ell}\alpha_\ell+1)}dt.
	\end{split}
\end{equation}
 Here coordinate transform $u_1=t_1,u_j=t_j-t_{j-1},j=2,\cdots,m$ and Lemma \ref{lmahx2020} gives that 
\[\E\big[L^{(\boldsymbol{\alpha})}_{\pm,\varepsilon}(T+h,x)-L^{(\boldsymbol{\alpha})}_{\pm,\varepsilon}(T,x)\big]^m\le c \sum_{\sS}h^{m-H\sum_{j=1}^{m}\sum_{\ell=1}^{d}(p_{j\ell}\alpha_\ell+\bar{p}_{j\ell}\alpha_\ell+1)}=c'h^{m(1-H(|\boldsymbol{\alpha}|+d))}.\]
Let $\varepsilon\downarrow0$ we get 
\[\E\big[L^{(\boldsymbol{\alpha})}_{\pm}(T+h,x)-L^{(\boldsymbol{\alpha})}_{\pm}(T,x)\big]^m\le c'h^{m(1-H(|\boldsymbol{\alpha}|+d))}.\]
Then using Kolmogorov continuity criterion we get locally $\theta_1-$H\"older continuity w.r.t. time variable $T$ for all $\theta_1\in\big(0,1-H(|\boldsymbol{\alpha}|+d)\big)$.  

For the space variable, consider $z\in\R$, from \eqref{fdalpha} and \eqref{approx}, 
\begin{align*}&L^{(\boldsymbol{\alpha})}_{\pm,\varepsilon}(T,z+x)-L^{(\boldsymbol{\alpha})}_{\pm,\varepsilon}(T,x)\\=&\frac{1}{(2\pi)^d}\int_{0}^{T}\int_{\R^d}\prod_{\ell=1}^{d}(\pm\iota y_{\ell})_0^{\alpha_\ell}e^{\iota y\cdot (x+X_t)}(e^{\iota y\cdot z}-1)e^{-\frac{\varepsilon}{2}|y|^2}dydt\\=&\sum_{k=1}^{d}\frac{1}{(2\pi)^d}\int_{0}^{T}\int_{\R^d}\prod_{\ell=1}^{d}(\pm\iota y_{\ell})_0^{\alpha_\ell}e^{\iota y\cdot (x+X_t)}e^{\iota\sum_{\ell=1}^{k-1}y_{\ell}z_{\ell}}(e^{\iota y_{k} z_k}-1)e^{-\frac{\varepsilon}{2}|y|^2}dydt=:\sum_{k=1}^{d}A_k,
\end{align*}
where $\E\big[L^{(\boldsymbol{\alpha})}_{\pm,\varepsilon}(T,z+x)-L^{(\boldsymbol{\alpha})}_{\pm,\varepsilon}(T,x)\big]^m\le d^m\sum\limits_{k=1}^{m}\E[A_k]^m$ for even $m$. Remember that for any $\alpha\in[0,1]$,$$|e^{-\iota y_k z_k}-1|\leq 2^{1-\alpha} |z_k|^{\alpha} |y_k|^{\alpha}\le 2^{1-\alpha}|y_k|^{\alpha} |z|^{\alpha},$$ like the arguments in \eqref{pthm11}, for even $m$,  
\begin{align*}
	&\E[A_k]^m\\\le&\frac{1}{(2\pi)^{md}}\int_{[0,T]^m}\int_{\R^{md}}\prod_{j=1}^{m}\prod_{\ell=1}^{d}|y_{j\ell}|^{\alpha_\ell}\\&\qquad\qquad\qquad\qquad\times\prod_{j=1}^{m}\big|e^{\iota y_{k} z_k}-1\big|\times\exp\bigg(-\frac12\sum_{\ell=1}^{d}\Var\Big(\sum_{j=1}^{m}y_{j\ell}X^{\ell}_{t_j}\Big)\bigg)dydt\\\le&\frac{m!2^{m(1-\alpha)}|z|^{m\alpha}}{(2\pi)^{md}}\int_{[0,T]^m_<}\int_{\R^{md}}\prod_{j=1}^{m}\prod_{\ell=1}^{d}|y_{j\ell}|^{\alpha_\ell+\alpha\eta_{k\ell}}\exp\bigg(-\frac12\sum_{\ell=1}^{d}\Var\Big(\sum_{j=1}^{m}y_{j\ell}X^{\ell}_{t_j}\Big)\bigg)dydt\\=&\frac{m!2^{m(1-\alpha)}|z|^{m\alpha}}{(2\pi)^{md}}\int_{[0,T]^m_<}\prod_{\ell=1}^{d}\bigg(\int_{\R^{m}}\prod_{j=1}^{m}|y_{j\ell}|^{\alpha_\ell+\alpha\eta_{k\ell}}\exp\bigg(-\frac12\Var\Big(\sum_{j=1}^{m}y_{j\ell}X^{\ell}_{t_j}\Big)\bigg)dy_{\cdot\ell}\bigg)dt.
\end{align*} 
  for each $k$ and $\alpha\in(0,1\wedge\frac12(\frac{1}{H}-2|\boldsymbol{\alpha}|-d))$, where $\eta_{ij}=\left\{\begin{array}{cl}
  	1,&i=j;\\
  	0,&i\ne j;
  \end{array}\right.$ in this article. Using Proposition \ref{upper1}, for even $m$,
\begin{align*}
	\E\big[L^{(\boldsymbol{\alpha})}_{\pm,\varepsilon}(T,z+x)&-L^{(\boldsymbol{\alpha})}_{\pm,\varepsilon}(T,x)\big]^m\\&\le c|z|^{m\alpha}\sum\limits_{k=1}^{m}\sum_{\sS} \int_{[0,T]^m_<}\prod_{\ell=1}^{d}\prod_{j=1}^{m}(t_j-t_{j-1})^{-H(p_{j\ell}(\alpha_\ell+\alpha\eta_{k\ell})+\bar{p}_{j\ell}(\alpha_\ell+\alpha\eta_{k\ell})+1)}dt\\&\le c|z|^{m\alpha},
\end{align*}
where the last inequality comes from the fact that 
 $$\displaystyle\sum_{\ell=1}^{d}H\big(p_{j\ell}(\alpha_\ell+\alpha\eta_{k\ell})+\bar{p}_{j\ell}(\alpha_\ell+\alpha\eta_{k\ell})+1\big)\le H(2|\boldsymbol{\alpha}|+2\alpha+d)<1$$ for any $p_{j\ell}$s and $\bar{p}_{j\ell}$s in $\sS$. So for these $m$, letting $\varepsilon\downarrow0$ we have $$\E\big[L^{(\boldsymbol{\alpha})}_{\pm,\varepsilon}(T,z+x)-L^{(\boldsymbol{\alpha})}_{\pm,\varepsilon}(T,x)\big]^m\le d^m\sum\limits_{k=1}^{m}\E[A_k]^m\le c|z|^{m\alpha},$$
 which means that by Kolmogorov continuity criterion, $L^{(\boldsymbol{\alpha})}_{\pm}(T,x)$ is locally $\theta_2-$H\"older continuous w.r.t. space variable $x$ for all $\theta_2\in(0,1\wedge\frac12(\frac{1}{H}-2|\boldsymbol{\alpha}|-d))$.
 
 In the last part of this section, we consider the convergence of $L^{(\boldsymbol{\beta})}_{\pm}(T,x)$ as $\beta_\ell\to\alpha_\ell$ for all $\ell=1,2,\cdots,d$. From \eqref{fdalpha} and \eqref{approx}, 
 \begin{align*}
 	&\quad L^{(\boldsymbol{\beta})}_{\pm,\varepsilon}(T,x)-L^{(\boldsymbol{\alpha})}_{\pm,\varepsilon}(T,x)\\&=\frac{1}{(2\pi)^d}\int_{0}^{T}\int_{\R^d}\Big(\prod_{\ell=1}^{d}(\pm\iota y_{\ell})_0^{\beta_\ell}-\prod_{\ell=1}^{d}(\pm\iota y_{\ell})_0^{\alpha_\ell}\Big)e^{\iota y\cdot (x+X_t)}e^{-\frac{\varepsilon}{2}|y|^2}dydt\\&=\sum_{k=1}^{d}\frac{1}{(2\pi)^d}\int_{0}^{T}\int_{\R^d}\Big(\prod_{\ell=k+1}^{d}(\pm\iota y_{\ell})_0^{\beta_\ell}\Big)\Big((\pm\iota y_{k})_0^{\beta_k}-(\pm\iota y_{k})_0^{\alpha_k}\Big)\Big(\prod_{\ell=1}^{k-1}(\pm\iota y_{\ell})_0^{\alpha_\ell}\Big)e^{\iota y\cdot (x+X_t)}e^{-\frac{\varepsilon}{2}|y|^2}dydt\\&=:\sum_{k=1}^{m}\widehat{A}_k,
 \end{align*} 
where for even $m$, $\E\big[L^{(\boldsymbol{\beta})}_{\pm,\varepsilon}(T,x)-L^{(\boldsymbol{\alpha})}_{\pm,\varepsilon}(T,x)\big]^m\le d^m\sum_{k=1}^{d}\E[\widehat{A}_k]^m$. Like \eqref{pthm11}, for these $m$,
\begin{align}
		&\E\big[	L^{(\boldsymbol{\beta})}_{\pm,\varepsilon}(T,x)-L^{(\boldsymbol{\alpha})}_{\pm,\varepsilon}(T,x)\big]^m\nonumber\\
	\le&\frac{m!}{(2\pi)^{md}}\sum_{k=1}^{d}\int_{[0,T]^m_<}\int_{\R^{md}}\Big(\prod_{\ell=k+1}^{d}\prod_{j=1}^{m}|y_{j\ell}|^{\beta_\ell}\Big)\prod_{j=1}^{m}\Big|(\pm\iota y_{jk})_0^{\beta_k}-(\pm\iota y_{jk})_0^{\alpha_k}\Big|\nonumber\\&\qquad\qquad\qquad\qquad\qquad\times\Big(\prod_{\ell=1}^{k-1}\prod_{j=1}^{m}|y_{j\ell}|^{\alpha_\ell}\Big)\times\exp\bigg(-\frac12\sum_{\ell=1}^{d}\Var\Big(\sum_{j=1}^{m}y_{j\ell}X^{\ell}_{t_j}\Big)\bigg)dydt\nonumber\\\le&c\int_{[0,T]^m_<}\int_{\R^{md}}\Big(\prod_{\ell=1}^{d}\prod_{j=1}^{m}(|y_{j\ell}|+1)^{\beta_\ell\vee\alpha_\ell}\Big)\times\exp\bigg(-\frac12\sum_{\ell=1}^{d}\Var\Big(\sum_{j=1}^{m}y_{j\ell}X^{\ell}_{t_j}\Big)\bigg)dydt\nonumber\\\le&c\int_{[0,T]^m_<}\prod_{\ell=1}^{d}\bigg(\int_{\R^{m}}\prod_{j=1}^{m}(|y_{j\ell}|+1)^{\alpha_\ell+\theta}\exp\bigg(-\frac12\Var\Big(\sum_{j=1}^{m}y_{j\ell}X^{\ell}_{t_j}\Big)\bigg)dy_{\cdot\ell}\bigg)dt,\label{upper2}
\end{align}
where $a\vee b:=\max\{a,b\}$ for all $a,b\in\R$, $\theta>0$ is a constant satisfying \[H\big(2\sum_{\ell=1}^d(\alpha_\ell+\theta)+d\big)<1\] and the last inequality exists when $\beta_\ell$ is close to $\alpha_\ell$ enough for all $\ell=1,2,\cdots,d$.
So by dominated convergence theorem it is enough to prove that 
\begin{align*}
	\int_{[0,T]^m_<}\prod_{\ell=1}^{d}\bigg(\int_{\R^{m}}\prod_{j=1}^{m}(|y_{j\ell}|+1)^{\alpha_\ell}\exp\bigg(-\frac12\Var\Big(\sum_{j=1}^{m}y_{j\ell}X^{\ell}_{t_j}\Big)\bigg)dy_{\cdot\ell}\bigg)dt<\infty
\end{align*}
for all $\boldsymbol{\alpha}=(\alpha_1,\cdots,\alpha_d)\in[0,\infty)^d$ with $H(2|\boldsymbol{\alpha}|+d)<1$. 

Using H\"older's inequality, Lemma \ref{Lem:cd1982}, \eqref{decomCov} and {\bf SLND} \eqref{slndp} gives that
\begin{align*}
	&\int_{\R^{m}}\prod_{j=1}^{m}(|y_{j\ell}|+1)^{\alpha_\ell}\exp\bigg(-\frac12\Var\Big(\sum_{j=1}^{m}y_{j\ell}X^{\ell}_{t_j}\Big)\bigg)dy_{\cdot\ell}\\\le&\prod_{j=1}^{m}\bigg(	\int_{\R^m}(|y_{j\ell}|+1)^{m\alpha_\ell}\exp\bigg(-\frac12\Var\Big(\sum_{j=1}^{m}y_{jk}X^{k}_{t_j}\Big)\bigg)dy_{\cdot k}\bigg)^{\frac1m}\\\le&c\Big(\det\Cov(X^1_{t_1},\cdots,X^1_{t_m})\Big)^{-\frac12}\times\prod_{j=1}^{m}\bigg(\int_{\R}\Big(\frac{|y_{j\ell}|}{\sigma_{t_j}}+1\Big)^{m\alpha_\ell}e^{-\frac{y_{j\ell}^2}{2}}dy_{j\ell}\bigg)^{\frac1m}\\\le&c\Big(\det\Cov(X^1_{t_1},\cdots,X^1_{t_m})\Big)^{-\frac12}\times\prod_{j=1}^{m}(\sigma_{t_j}^{-\alpha_\ell}+1)\\\le&c\Big(\det\Cov(X^1_{t_1},\cdots,X^1_{t_m})\Big)^{-\frac12}\times\prod_{j=1}^{m}(\widetilde{\sigma}_{t_j}^{-\alpha_\ell}+1)\\\le&c\Big(\det\Cov(X^1_{t_1},\cdots,X^1_{t_m})\Big)^{-\frac12}\times\prod_{j=1}^{m}\widetilde{\sigma}_{t_j}^{-\alpha_\ell}\\\le&c\sum_{\sS_\ell}\prod_{j=1}^{m}(t_j-t_{j-1})^{-H(p_{j\ell}\alpha_\ell+\bar{p}_{j\ell}\alpha_\ell+1)}<\infty,
\end{align*}
for any $0=t_0<t_1<t_2<\cdots<t_m<T$, where $$\widetilde{\sigma}_{t_j}=\left\{\begin{array}{cl}
	\min\{(t_j-t_{j-1})^{2H},(t_{j+1}-t_j)^{2H}\}, & j=1,\cdots,m-1,\\ (t_m-t_{m-1})^{2H}, & j=m.	
\end{array}\right.,$$ the third inequality is because 
\[\int_{\R}\Big(\frac{|y_{j\ell}|}{\sigma_{t_j}}+1\Big)^{m\alpha_\ell}e^{-\frac{y_{j\ell}^2}{2}}dy_{j\ell}\le2^{m\alpha_\ell}\Big(\sigma_{t_j}^{-m\alpha_\ell}\int_{\R}|v|^{m\alpha_\ell}e^{-\frac{v^2}{2}}dv+\int_{\R}e^{-\frac{v^2}{2}}dv\Big)\]
and the fifth inequality comes from the fact $\widetilde{\sigma}_{t_j}\le T^{2H}$ for $j=1,2,\cdots,m.$
Then we get 
\begin{align*}
	\int_{[0,T]^m_<}\prod_{\ell=1}^{d}\bigg(\int_{\R^{m}}\prod_{j=1}^{m}(|y_{j\ell}|&+1)^{\alpha_\ell}\exp\bigg(-\frac12\Var\Big(\sum_{j=1}^{m}y_{j\ell}X^{\ell}_{t_j}\Big)\bigg)dy_{\cdot\ell}\bigg)dt\\&\le c\sum_{\sS} \int_{[0,T]^m_<}\prod_{j=1}^{m}(t_j-t_{j-1})^{-H\sum_{\ell=1}^{d}(p_{j\ell}\alpha_\ell+\bar{p}_{j\ell}\alpha_\ell+1)}dt<\infty,
\end{align*}
where the last inequality is from 
\[\sum_{\ell=1}^{d}H(p_{j\ell}\alpha_\ell+\bar{p}_{j\ell}\alpha_\ell+1)\le H(2\sum_{\ell=1}^{d}|\boldsymbol{\alpha}|+d)<1.\]
 So letting $\varepsilon\downarrow0$ and using dominated convergence theorem in \eqref{upper2} we have 
\[\lim\limits_{\beta_\ell\to\alpha_\ell,\ell=1,\cdots,d}\mathrm{E}\big[L^{(\boldsymbol{\beta})}_{\pm}(T,x)-L^{(\boldsymbol{\alpha})}_{\pm}(T,x)\big]^m=0\]
for even $m$, which completes the proof.
\section{Proof of Theorem \ref{nes}}
In this section, we focus on the proof of Theorem \ref{nes}, where for simplicity we denote
\begin{align*}
	A&=A(t)=\E[X^1_t]^2+\varepsilon,\\B&=B(t,s)=\E[X^1_tX^1_s],\\C&=C(s)=\E[X^1_s]^2+\varepsilon,\\D&=D(t,s)=A-2B+C,\\G&=G(t,s)=\sqrt{\frac{A-2B+C}{AC-B^2}},\\\Delta&=\Delta(t,s)=AC-B^2
\end{align*}
for all $t,s>0$ and $\varepsilon>0$ in this section. 
Now we consider the proof of Theorem \ref{nes}. From \eqref{fdalpha} and \eqref{approx}, 
\begin{align}
	\E\big[L^{(\boldsymbol{\alpha})}_{\pm,\varepsilon}(T,x)\big]^2&=\frac{2}{(2\pi)^{2d}}\int_{[0,T]^2_<}\int_{\R^{2d}}\prod_{\ell=1}^{d}(\pm\iota y_{1\ell})_0^{\alpha_\ell}(\pm\iota y_{2\ell})_0^{\alpha_\ell}\times e^{\iota(y_1+y_2)\cdot x}\nonumber\\&\qquad\times\exp\Big(-\frac{1}{2}\sum_{\ell=1}^{d}\Var\Big(y_{1\ell}X^\ell_{t}+y_{1\ell}X^\ell_{s}\Big)-\frac{\varepsilon}{2}(|y_1|^2+|y_2|^2)\Big)dydtds\nonumber\\&=\frac{2}{(2\pi)^{2d}}\int_{[0,T]^2_<}\prod_{\ell=1}^{d}\bigg(\int_{\R^{2}}(\pm\iota v_1)^{\alpha_\ell}_0(\pm\iota v_2)^{\alpha_\ell}_0\times e^{\iota(v_1+v_2)x_\ell}\nonumber\\&\qquad\qquad\qquad\qquad\qquad\times\exp\Big(-\frac{1}{2}\big(Av_1^2+2Bv_1v_2+Cv_2^2\big)\Big)dv_1dv_2\bigg)dtds\nonumber\\&=:\frac{2}{(2\pi)^{2d}}\int_{[0,T]^2_<}\prod_{\ell=1}^{d}J_{\pm}(\alpha_\ell,x_\ell,t,s)dtds,\label{exosm}
\end{align}
where $[0,T]^2_<=\{(t,s):0<s<t<T\}$. Because $\E\big[L^{(\boldsymbol{\alpha})}_{+,\varepsilon}(T,x)\big]^2=\E\big[L^{(\boldsymbol{\alpha})}_{-,\varepsilon}(T,-x)\big]^2$, it is enough to prove that $$\liminf_{\varepsilon\downarrow0}\E\big[L^{(\boldsymbol{\alpha})}_{+,\varepsilon}(T,x)\big]^2=+\infty$$
in the proof. Using coordinate transform 
\begin{align*}
	v_1&=-\frac{u_1}{\sqrt{A-2B+C}}+\frac{(A-B)u_2}{\sqrt{(AC-B^2)(A-2B+C)}},\\
	v_2&=+\frac{u_1}{\sqrt{A-2B+C}}+\frac{(C-B)u_2}{\sqrt{(AC-B^2)(A-2B+C)}},
\end{align*}
we get 
\begin{align*}
	J_+(\alpha_\ell,x_\ell,t,s)=\frac{G}{D^{\frac{2\alpha_\ell+1}{2}}}\int_{\R^2}\big(-&\iota u_1-\iota\frac{C-B}{\sqrt{\Delta}} u_2+\iota\sqrt{D}Gu_2\big)^{\alpha_\ell}_0\\\times\big(&\iota u_1+\iota \frac{C-B}{\sqrt{\Delta}} u_2\big)^{\alpha_\ell}_0e^{-\frac12u_1^2-\frac12u_2^2+\iota Gu_2x_\ell}du_1du_2.
\end{align*}
Moreover, we denote 
\begin{align*}
	\widehat{J}_+(\alpha_\ell,x_\ell,t,s)=\frac{G}{D^{\frac{2\alpha_\ell+1}{2}}}\int_{\R^2}\big(-&\iota u_1-\iota\frac{C-B}{\sqrt{\Delta}} u_2\big)^{\alpha_\ell}_0\\\times\big(&\iota u_1+\iota \frac{C-B}{\sqrt{\Delta}} u_2\big)^{\alpha_\ell}_0e^{-\frac12u_1^2-\frac12u_2^2+\iota Gu_2x_\ell}du_1du_2
\end{align*}
and 
\begin{align*}
	\widetilde{J}_+(\alpha_\ell,x_\ell,t,s)=\frac{G}{D^{\frac{2\alpha_\ell+1}{2}}}\int_{\R^2}\big(-\iota u_1\big)^{\alpha_\ell}_0\times\big(\iota u_1\big)^{\alpha_\ell}_0e^{-\frac12u_1^2-\frac12u_2^2+\iota Gu_2x_\ell}du_1du_2.
\end{align*}

The rest of this section is divided into three parts.

{\bf Part I: The case that $x=0$, $Hd\le1$}

In this part, we assume $X\in G^{d,H}_{S,U}$. To obtain $\liminf_{\varepsilon\downarrow0}\E\big[L^{(\boldsymbol{\alpha})}_{+,\varepsilon}(T,0)\big]^2=+\infty$, define
$$\widehat{I}_{+,\varepsilon}^{(\boldsymbol{\alpha})}(0):=\frac{2}{(2\pi)^{2d}}\int_{[0,T]^2_<}\prod_{\ell=1}^{d}\widehat{J}_{+}(\alpha_\ell,0,t,s)dtds.$$
Then we will prove that there exists a positive function $f_{\boldsymbol{\alpha}}(\varepsilon)$ converging to 0 as $\varepsilon\downarrow0$ and a positive constant $c_{\boldsymbol{\alpha}}$, s.t. 
\[\lim_{\varepsilon\downarrow0}f_{\boldsymbol{\alpha}}(\varepsilon)\Big|\E\big[L^{(\boldsymbol{\alpha})}_{+,\varepsilon}(T,0)\big]^2-\widehat{I}_{+,\varepsilon}^{(\boldsymbol{\alpha})}(0)\Big|=0\hbox{ and }\liminf_{\varepsilon\downarrow0}f_{\boldsymbol{\alpha}}(\varepsilon)\widehat{I}_{+,\varepsilon}^{(\boldsymbol{\alpha})}(0)\ge c_{\boldsymbol{\alpha}}.\]
 Before we start the proof, we need to verify the following propositions about the notations defined at the beginning of this section. 

\begin{proposition}\label{acd}
	For $X\in G^{d,H}_{S,U}$, we have
	\begin{align*}
		(\kappa_H\wedge1)(t^{2H}+\varepsilon)&\le A\le(K_H\vee1)(t^{2H}+\varepsilon),\\(\kappa_H\wedge1)(s^{2H}+\varepsilon)&\le C\le(K_H\vee1)(s^{2H}+\varepsilon),\\
		(\kappa_{H}\wedge2)\big((t-s)^{2H}+\varepsilon\big)&\le D\le(K_H\vee2)\big((t-s)^{2H}+\varepsilon\big)
	\end{align*} 
	for all $t,s\ge0$ and $\varepsilon>0$.
\end{proposition}
\begin{proof}
	Without loss of generality, assume $t>s\ge0$. The result of the proposition comes from the fact 
	\[\kappa_H(t-s)^{2H}\le\E(X^1_t-X^1_s)^2\le K_H(t-s)^{2H}\] 
	where in the right hand side of the inequalities we use \eqref{smoi} and in the left hand side of the inequalities we use {\bf SLND} \eqref{slndp} and the fact $\E(X^1_t-X^1_s)^2\ge\Var(X^1_t|X^1_s)$.
\end{proof}
\begin{proposition}\label{dgc-b}
	For $X\in G^{d,H}_{S,U}$, we have 
	\begin{equation}\label{delta}
		(\kappa_H\wedge1)^2\big((t-s)^{2H}+\varepsilon\big)\big(s^{2H}+\varepsilon\big)\le\Delta\le(K_H\vee2)^2\big(s^{2H}+\varepsilon\big)\big((t-s)^{2H}+\varepsilon\big),
	\end{equation}
	\begin{equation}\label{g}
		\sqrt{\frac{\kappa_H\wedge2}{(K_H\vee2)^2}}\cdot \big(s^{2H}+\varepsilon\big)^{-\frac{1}{2}}\le	G\le\sqrt{\frac{K_H\vee2}{(\kappa_H\wedge1)^2}}\cdot \big(s^{2H}+\varepsilon\big)^{-\frac{1}{2}} 
	\end{equation}
	\begin{equation}\label{c-b}
		|C-B|\le(K_H\vee1)\big(s^{2H}+\varepsilon\big)^{\frac12}\big((t-s)^{2H}+\varepsilon\big)^{\frac12}
	\end{equation}
	for all $t>s>0$ and $\varepsilon>0$.
\end{proposition}
\begin{proof}
	Using \eqref{smoi}, we get 
	\begin{equation}
		\begin{split}
			|C-B|&=\big|\E \big[X^1_s(X^1_s-X^1_t)\big]+\varepsilon\big|\\&\le\big[\E(X^1_s)^2\big]^{\frac12}\big[\E(X^1_t-X^1_s)^2\big]^{\frac12}+\varepsilon\\&\le\big(\E(X^1_s)^2+\varepsilon\big)^{\frac12}\big(\E(X^1_t-X^1_s)^2+\varepsilon\big)^{\frac12}\\&\le(K_H\vee1)\big(s^{2H}+\varepsilon\big)^{\frac12}\big((t-s)^{2H}+\varepsilon\big)^{\frac12}.
		\end{split}
	\end{equation}
	Using \eqref{smoi} and {\bf SLND} \eqref{slndp}, we get 
	\begin{equation}\label{delta1}
		\begin{split}
			\Delta&=AC-B^2\\&=\det\Cov(X^1_t,X^1_s)+\big(\E(X^1_t)^2+\E(X^1_s)^2\big)\varepsilon+\varepsilon^2\\&=\Var(X^1_t|X^1_s)\E(X^1_s)^2+\big(\E(X^1_t)^2+\E(X^1_s)^2\big)\varepsilon+\varepsilon^2\\&\ge(\kappa_{H})^2(t-s)^{2H}s^{2H}+\kappa_{H}\big(t^{2H}+s^{2H}\big)\varepsilon+\varepsilon^2\\&\ge(\kappa_{H})^2(t-s)^{2H}s^{2H}+\kappa_{H}\big((t-s)^{2H}+s^{2H}\big)\varepsilon+\varepsilon^2\\&=\big(\kappa_{H}(t-s)^{2H}+\varepsilon\big)\big(\kappa_{H}s^{2H}+\varepsilon\big)\\&\ge(\kappa_H\wedge1)^2\big((t-s)^{2H}+\varepsilon\big)\big(s^{2H}+\varepsilon\big).
		\end{split}
	\end{equation}
	Using Proposition \ref{acd} we get 
	\begin{equation}\label{delta2}
		\begin{split}
			\Delta=AC-B^2&=(D+2(B-C)+C)C-(C+B-C)^2\\&\le DC\le(K_H\vee2)^2\big(s^{2H}+\varepsilon\big)\big((t-s)^{2H}+\varepsilon\big).
		\end{split}
	\end{equation}
	Using \eqref{delta1}, \eqref{delta2} and Proposition \ref{acd} we get 
	\begin{equation*}
		\begin{split}
			G=\sqrt{\frac{D}{AC-B^2}}\le\sqrt{\frac{(K_H\vee2)\big((t-s)^{2H}+\varepsilon\big)}{(\kappa_H\wedge1)^2\big((t-s)^{2H}+\varepsilon\big)\big(s^{2H}+\varepsilon\big)}}\le\sqrt{\frac{K_H\vee2}{(\kappa_H\wedge1)^2}}\cdot \big(s^{2H}+\varepsilon\big)^{-\frac{1}{2}}
		\end{split}
	\end{equation*}	
	and 
	\begin{equation*}
		\begin{split}
			G=\sqrt{\frac{D}{AC-B^2}}\ge\sqrt{\frac{(\kappa_H\wedge2)\big((t-s)^{2H}+\varepsilon\big)}{(K_H\vee2)^2\big((t-s)^{2H}+\varepsilon\big)\big(s^{2H}+\varepsilon\big)}}\ge\sqrt{\frac{\kappa_H\wedge2}{(K_H\vee2)^2}}\cdot \big(s^{2H}+\varepsilon\big)^{-\frac{1}{2}}.
		\end{split}
	\end{equation*}
\end{proof}

Then we continue the proof of statement {\rm(1)} in Theorem \ref{nes}. 

For the case $|\boldsymbol{\alpha}|=0$, using coordinate transform $u=t-s,v=s$, Proposition \ref{acd} and \ref{dgc-b} we get 
\begin{align*}
	\E\big[L^{(\boldsymbol{0})}_{+,\varepsilon}(T,0)\big]^2&=\frac{2}{(2\pi)^{2d}}\int_{[0,T]^2_<}\prod_{\ell=1}^{d}J_{+}(0,0,t,s)dtds\\&\ge c\int_{[0,T]^2_<}\frac{G^d}{D^{\frac{d}{2}}}dtds\ge c\int_{0}^{\frac{T}{2}}\big(u^{2H}+\varepsilon\big)^{-\frac{d}{2}}du\int_{0}^{\frac{T}{2}}\big(v^{2H}+\varepsilon\big)^{-\frac{d}{2}}dv,
\end{align*}
where using Lemma \ref{lmaxiao}, when $Hd=1$,  there exists 
\[\E\big[L^{(\boldsymbol{0})}_{+,\varepsilon}(T,0)\big]^2\ge c\ln^2(1+\varepsilon^{-\frac12})\]
for any $\varepsilon\in(0,\frac12)$, which leads to the desired result. 

For the case $|\boldsymbol{\alpha}|>0$, consider 
\begin{align*}
	J_+(\alpha_\ell,0,t,s)=\frac{G}{D^{\frac{2\alpha_\ell+1}{2}}}\int_{\R^2}\big(-&\iota u_1-\iota\frac{C-B}{\sqrt{\Delta}} u_2+\iota\sqrt{D}Gu_2\big)^{\alpha_\ell}_0\\\times\big(&\iota u_1+\iota \frac{C-B}{\sqrt{\Delta}} u_2\big)^{\alpha_\ell}_0e^{-\frac12u_1^2-\frac12u_2^2}du_1du_2
\end{align*}
and 
\begin{align*}
	\widehat{J}_+(\alpha_\ell,0,t,s)=\frac{G}{D^{\frac{2\alpha_\ell+1}{2}}}\int_{\R^2}\big(-&\iota u_1-\iota\frac{C-B}{\sqrt{\Delta}} u_2\big)^{\alpha_\ell}_0\\\times\big(&\iota u_1+\iota \frac{C-B}{\sqrt{\Delta}} u_2\big)^{\alpha_\ell}_0e^{-\frac12u_1^2-\frac12u_2^2}du_1du_2.
\end{align*}
Using Lemma \ref{diff} and Propositions \ref{acd}, \ref{dgc-b}, we get 
\begin{align*}
	|\widehat{J}_+(\alpha_\ell,0,t,s)|\le&c\frac{G}{D^{\frac{2\alpha_\ell+1}{2}}}\int_{\R^2}\Big(|u_1|^{\alpha_\ell}+\Big|\frac{C-B}{\sqrt{\Delta}}\Big|^{\alpha_\ell}|u_2|^{\alpha_\ell}\Big)^2e^{-\frac12u_1^2-\frac12u_2^2}du_1du_2\\\le&c\big(\big(t-s)^{2H}+\varepsilon\big)^{-\frac{2\alpha_\ell+1}{2}}\big(s^{2H}+\varepsilon\big)^{-\frac{1}{2}}
\end{align*}
and
\begin{align*}
	&|J_+(\alpha_\ell,0,t,s)-\widehat{J}_+(\alpha_\ell,0,t,s)|\\\le&c\frac{G}{D^{\frac{2\alpha_\ell+1}{2}}}\int_{\R^2}\Big(\big|\sqrt{D}Gu_2\big|^{\alpha^*_\ell}\big|u_1+\frac{C-B}{\sqrt{\Delta}}u_2\big|^{2\alpha_\ell-\alpha^{*}_\ell}\\&\qquad\qquad\qquad\qquad\qquad\qquad+\big|\sqrt{D}Gu_2\big|^{\alpha_\ell}\big|u_1+\frac{C-B}{\sqrt{\Delta}}u_2\big|^{\alpha_\ell}\Big)e^{-\frac12u_1^2-\frac12u_2^2}du_1du_2\\\le&c\frac{G}{D^{\frac{2\alpha_\ell+1}{2}}}\int_{\R^2}\Big(2^{2\alpha_\ell-\alpha^{*}_\ell}|DG^2|^{\frac{\alpha^*_\ell}{2}}|u_2|^{\alpha^*_\ell}|u_1|^{2\alpha_\ell-\alpha^{*}_\ell}+2^{2\alpha_\ell-\alpha^{*}_\ell}\Big|\frac{C-B}{\sqrt{\Delta}}\Big|^{2\alpha_\ell-\alpha^{*}_\ell}|DG^2|^{\frac{\alpha^*_\ell}{2}}|u_2|^{2\alpha_\ell}\\&\qquad\qquad\qquad\qquad+2^{\alpha_\ell}|DG^2|^{\frac{\alpha_\ell}{2}}|u_1|^{\alpha_\ell}|u_2|^{\alpha_\ell}+2^{\alpha_\ell}|DG^2|^{\frac{\alpha_\ell}{2}}\Big|\frac{C-B}{\sqrt{\Delta}}\Big|^{\alpha_\ell}|u_2|^{2\alpha_\ell}\Big)e^{-\frac12u_1^2-\frac12u_2^2}du_1du_2\\\le& c\Big(D^{-\frac{2\alpha_\ell-\alpha^*_\ell+1}{2}}G^{1+\alpha^*_\ell}+D^{-\frac{\alpha_\ell+1}{2}}G^{1+\alpha_\ell}\Big)\\\le&c\Big(\big(\big(t-s)^{2H}+\varepsilon\big)^{-\frac{2\alpha_\ell-\alpha^*_\ell+1}{2}}\big(s^{2H}+\varepsilon\big)^{-\frac{1+\alpha^*_\ell}{2}}+\big((t-s)^{2H}+\varepsilon\big)^{-\frac{\alpha_\ell+1}{2}}\big(s^{2H}+\varepsilon\big)^{-\frac{1+\alpha_\ell}{2}}\Big),
\end{align*}
where for all $\alpha\ge0$, $\alpha^*$ is defined in \eqref{alphastar} and we get  $0<\alpha^*=\widetilde{\alpha}+(\bar{\alpha}\wedge1)\mathbf{1}_{\{\widetilde{\alpha}=0\}}<\alpha$ for all $\alpha>0$ and $\alpha^{*}=0$ for $\alpha=0$.
So that
\begin{align}\label{thm211}
	&\int_{[0,T]^2_<}\Big|\prod_{\ell=1}^{d}J_+(\alpha_\ell,0,t,s)-\prod_{\ell=1}^{d}\widehat{J}_+(\alpha_\ell,0,t,s)\Big|dtds\nonumber\\\le&c\sum_{k=1:\alpha_k>0}^{d}\int_{[0,T]^2_<}\Big|J_+(\alpha_k,0,t,s)-\widehat{J}_+(\alpha_k,0,t,s)\Big|\nonumber\\&\qquad\qquad\qquad\times\prod_{\ell=1,\ell\ne k}\Big(|J_+(\alpha_\ell,0,t,s)-\widehat{J}_+(\alpha_\ell,0,t,s)|+|\widehat{J}_+(\alpha_\ell,0,t,s)|\Big)dtds\nonumber\\\le&c\sum_{k=1:\alpha_k>0}^{d}\sum_{\widehat{\sS}_k}\int_{[0,T]^2_<}\big((t-s)^{2H}+\varepsilon\big)^{-\sum_{\ell=1}^{d}\frac{2\alpha_\ell-(q_\ell\alpha_\ell+\bar{q}_\ell\alpha^*_\ell)+1}{2}}\big(s^{2H}+\varepsilon\big)^{-\sum_{\ell=1}^{d}\frac{1+(q_\ell\alpha_\ell+\bar{q}_\ell\alpha^*_\ell)}{2}}dtds\nonumber\\\le&c\sum_{k=1:\alpha_k>0}^{d}\sum_{\widehat{\sS}_k}\int_{0}^T\big(u^{2H}+\varepsilon\big)^{-\sum_{\ell=1}^{d}\frac{2\alpha_\ell-(q_\ell\alpha_\ell+\bar{q}_\ell\alpha^*_\ell)+1}{2}}du\int_{0}^T\big(v^{2H}+\varepsilon\big)^{-\sum_{\ell=1}^{d}\frac{1+(q_\ell\alpha_\ell+\bar{q}_\ell\alpha^*_\ell)}{2}}dv\nonumber\\\le&\left\{\begin{array}{cl}
		c\varepsilon^{\frac{1}{2H}-|\boldsymbol{\alpha}|-\frac{d}{2}},	&Hd=1,H(2|\boldsymbol{\alpha}|+d)>1,\\c\varepsilon^{\frac{1}{2H}-|\boldsymbol{\alpha}|-\frac{d}{2}+\beta},&Hd<1,H(2|\boldsymbol{\alpha}|+d)>1,\\c, &Hd<1,H(2|\boldsymbol{\alpha}|+d)=1,
	\end{array}\right.
\end{align} 
for any $\varepsilon\in(0,\frac12)$, where $$\widehat{\sS}_k=\{q_\ell,\bar{q}_\ell,\tilde{q}_\ell\in\{0,1\}:\tilde{q}_k=0,q_\ell+\bar{q}_\ell+\tilde{q}_\ell=1,\ell=1,\cdots,d\},$$
\[\beta=\frac{1}{4}\Big\{\min\{\alpha_\ell^{*}:\alpha_\ell^{*}>0,\ell=1,\cdots,d\}\wedge\big(\frac{1}{H}-d\big)\wedge\big(2|\boldsymbol{\alpha}|+d-\frac{1}{H}\big)\Big\},\] the third inequality comes from coordinate transform $u=t-s,v=s$ and using Lemma \ref{lmaxiao} we can obtain the last inequality. Moreover, for $\widehat{J}_{+}(\alpha_\ell,0,t,s)$, using Proposition \ref{acd}, Proposition \ref{dgc-b} and coordinate transform $u=t-s,v=s$,
\begin{align}
		\widehat{I}_{+,\varepsilon}^{(\boldsymbol{\alpha})}(0)=&\frac{2}{(2\pi)^{2d}}\int_{[0,T]^2_<}\prod_{\ell=1}^{d}\widehat{J}_{+}(\alpha_\ell,0,t,s)dtds\nonumber\\=&\frac{2}{(2\pi)^{2d}}\int_{[0,T]^2_<}\prod_{\ell=1}^{d}\bigg(\frac{G}{D^{\frac{2\alpha_\ell+1}{2}}}\int_{\R^2} |u_2|^{2\alpha_\ell}\exp\Big(-\frac12\big(u_1-\frac{C-B}{\sqrt{\Delta}} u_2\big)^2-\frac12u_2^2\Big)du_1du_2\bigg)dtds\nonumber\\\ge& c\int_{[0,T]^2_<}\prod_{\ell=1}^{d}\bigg(\frac{G}{D^{\frac{2\alpha_\ell+1}{2}}}\int_{\R^2} |u_2|^{2\alpha_\ell}\exp\Big(-\frac{c_1}2(u_1^2+u_2^2)\Big)du_1du_2\bigg)dtds\nonumber\\\ge& c\int_{0}^{\frac{T}{2}}\big(u^{2H}+\varepsilon\big)^{-\frac{2|\boldsymbol{\alpha}|+d}{2}}du\int_{0}^{\frac{T}{2}}\big(v^{2H}+\varepsilon\big)^{-\frac{d}{2}}dv\nonumber\\\ge&\left\{\begin{array}{cl}
	c\ln(1+\varepsilon^{-\frac12})\varepsilon^{\frac{1}{2H}-|\boldsymbol{\alpha}|-\frac{d}{2}},	&Hd=1,H(2|\boldsymbol{\alpha}|+d)>1,\\c\varepsilon^{\frac{1}{2H}-|\boldsymbol{\alpha}|-\frac{d}{2}},&Hd<1,H(2|\boldsymbol{\alpha}|+d)>1,\\c\ln(1+\varepsilon^{-\frac12}),&Hd<1,H(2|\boldsymbol{\alpha}|+d)=1,
\end{array}\right.\label{thm212}
\end{align}
for any $\varepsilon\in(0,\frac12)$. Then combining \eqref{thm211} with \eqref{thm212} we can get 
\[\E\big[L^{(\boldsymbol{\alpha})}_{+,\varepsilon}(T,0)\big]^2\\\ge\left\{\begin{array}{cl}
	c\ln(1+\varepsilon^{-\frac12})\varepsilon^{\frac{1}{2H}-|\boldsymbol{\alpha}|-\frac{d}{2}},	&Hd=1,H(2|\boldsymbol{\alpha}|+d)>1,\\c\varepsilon^{\frac{1}{2H}-|\boldsymbol{\alpha}|-\frac{d}{2}},&Hd<1,H(2|\boldsymbol{\alpha}|+d)>1,\\c\ln(1+\varepsilon^{-\frac12}),&Hd<1,H(2|\boldsymbol{\alpha}|+d)=1,
\end{array}\right.\]
for any $\varepsilon\in(0,\frac12)$, which completes the proof of {\rm(1)} in Theorem \ref{nes}.

{\bf Part II: The case $x=0$, $Hd>1$}

In this part, we assume $X\in {G}^{d,H}_{S,U}$ for  $\sum\limits_{\ell=1,\alpha_\ell\in\Z}^{d}\alpha_\ell$ even and $X\in \widehat{G}^{d,H}_{S,U}$ for $\sum\limits_{\ell=1,\alpha_\ell\in\Z}^{d}\alpha_\ell$ odd. Denote $I=\{\ell:\ell=1,2,\cdots,d,\alpha_\ell\in\Z\}$ to be the set of $\ell$ s.t. $\alpha_\ell$ is integer. To obtain $\liminf_{\varepsilon\downarrow0}\E\big[L^{(\boldsymbol{\alpha})}_{+,\varepsilon}(T,0)\big]^2=+\infty$, we use the fact 
\[\E\big[L^{(\boldsymbol{\alpha})}_{+,\varepsilon}(T,0)\big]^2=\E\bigg(\E\Big[\big[L^{(\boldsymbol{\alpha})}_{+,\varepsilon}(T,0)\big]^2\Big|X^\ell,\ell\in I\Big]\bigg)\ge\E\Big[\E\big(L^{(\boldsymbol{\alpha})}_{+,\varepsilon}(T,0)|X^\ell,\ell\in I\big)\Big]^2,\]
which comes from Jensen's inequality and prove $$\liminf_{\varepsilon\downarrow0}\E\Big[\E\big(L^{(\boldsymbol{\alpha})}_{+,\varepsilon}(T,0)|X^\ell,\ell\in I\big)\Big]^2=+\infty.$$
Before we start the proof, we need to verify the following proposition: 

\begin{proposition}\label{b}
	For $X\in G^{d,H}_{S,U}$, we have $	B\ge\frac{\kappa_H}{2}s^{2H}$ for any $0<s<t$ with $\frac{t-s}{s}\le\sqrt[H]{\frac{2\kappa_H}{K_H}}$.
\end{proposition}
\begin{proof}
	Using \eqref{smoi}, {\bf SLND} \eqref{slndp} and the fact $\E(X^1_s)^2\ge\Var(X^1_s|X^1_{2s})\ge\kappa_Hs^{2H}$, we get 
	\begin{align*}
		\E(X^1_tX^1_s)=&\E\big[(X^1_t-X^1_s)X^1_s\big]+\E(X^1_s)^2\\\ge&\E(X^1_s)^2-\big(\E(X^1_t-X^1_s)^2\big)^{\frac12}\big(\E(X^1_s)^2\big)^{\frac12}\\\ge&\kappa_Hs^{2H}-K_H(t-s)^Hs^H\\\ge&\kappa_Hs^{2H}-K_H\Big(\frac{\kappa_H}{2K_H}s^H\Big)s^H=\frac{\kappa_H}{2}s^{2H}
	\end{align*}
	for any $0<s<t$ with $\frac{t-s}{s}\le\sqrt[H]{\frac{\kappa_H}{2K_H}}$.
\end{proof}

From \eqref{fdalpha} and \eqref{approx}, 
\begin{align*}
	L^{(\boldsymbol{\alpha})}_{+,\varepsilon}(T,0)=\frac{1}{(2\pi)^d}\int_{0}^{T}\int_{\R^d}&\prod_{\ell=1}^{d}(\iota y_{\ell})_0^{\alpha_\ell}e^{\iota y\cdot X_t}e^{-\frac{\varepsilon}{2}|y|^2}dydt,\\
	\E\big(L^{(\boldsymbol{\alpha})}_{+,\varepsilon}(T,0)|X^\ell,\ell\in I\big)=\frac{1}{(2\pi)^d}\int_{0}^{T}&\prod_{\ell\in I}\Big(\int_{\R^d}(\iota y_{\ell})_0^{\alpha_\ell}e^{\iota y_\ell X_t^\ell}e^{-\frac{\varepsilon}{2}|y_\ell|^2}dy_\ell \Big)\\\times&\prod_{\ell\notin I}\Big(\int_{\R^d}(\iota y_{\ell})_0^{\alpha_\ell}e^{-\frac{\varepsilon}{2}|y_\ell|^2(\E(X_t^\ell)^2+\varepsilon)}dy_\ell \Big) dt
\end{align*}
and 
\begin{align*}
	&\E\Big[\E\big(L^{(\boldsymbol{\alpha})}_{+,\varepsilon}(T,0)|X^\ell,\ell\in I\big)\Big]^2\\=&\frac{2}{(2\pi)^{2d}}\int_{[0,T]^2_<}\prod_{\ell\in I}\bigg(\int_{\R^{2}}(-1)^{\alpha_\ell}v_1^{\alpha_\ell}v_2^{\alpha_\ell}\exp\Big(-\frac{1}{2}\big(Av_1^2+2Bv_1v_2+Cv_2^2\big)\Big)dv_1dv_2\bigg)\\&\qquad\qquad\qquad\times\prod_{\ell\notin I}\bigg(\int_{\R^{2}}(\iota v_1)^{\alpha_\ell}_0(\iota v_2)^{\alpha_\ell}_0\exp\Big(-\frac{1}{2}\big(Av_1^2+Cv_2^2\big)\Big)dv_1dv_2\bigg)dtds.
\end{align*}

From Lemma \ref{lmahx} and the fact that $\sum\limits_{\ell=1,\alpha_\ell\in\Z}^{d}\alpha_\ell$ is even or $B\ge0$ for $X\in \widehat{G}^{d,H}_{S,U}$, we get 
\begin{equation}\label{Ie}
	\begin{split}
		&\prod_{\ell\in I}\bigg(\int_{\R^{2}}(-1)^{\alpha_\ell}v_1^{\alpha_\ell}v_2^{\alpha_\ell}\exp\Big(-\frac{1}{2}\big(Av_1^2+2Bv_1v_2+Cv_2^2\big)\Big)dv_1dv_2\bigg)\\=&\prod_{\ell\in I}\bigg(\sum_{m=0,\text{m even}}^{\alpha_\ell}\frac{c_{\alpha_\ell,m}B^{\alpha_\ell-m}}{(AC-B^2)^{\frac{2\alpha_\ell-m+1}{2}}}\bigg)\ge\prod_{\ell\in I}\frac{c_{\alpha_\ell,0}B^{\alpha_\ell}}{(AC-B^2)^{\frac{2\alpha_\ell+1}{2}}},
	\end{split}
\end{equation}
where $c_{\alpha_\ell,0}=\frac{\alpha_\ell!(2\alpha_\ell-1)!!}{(\alpha_\ell)!}=(2\alpha_\ell-1)!!$ for $\ell\in I$. Moreover,  
\begin{align*}
	&\prod_{\ell\notin I}\bigg(\int_{\R^{2}}(\iota v_1)^{\alpha_\ell}_0(\iota v_2)^{\alpha_\ell}_0\exp\Big(-\frac{1}{2}\big(Av_1^2+Cv_2^2\big)\Big)dv_1dv_2\bigg)\\=&\prod_{\ell\notin I}(AC)^{-\frac{\alpha_\ell+1}{2}}\Big(\int_{\R}|v|^{\alpha_\ell}e^{\iota\frac{\pi\alpha_\ell}{2}\sgn(v)}e^{-\frac{1}{2}v^2}dv\Big)^2\\=&\prod_{\ell\notin I}(AC)^{-\frac{\alpha_\ell+1}{2}}\Big(\int_{\R}|v|^{\alpha_\ell}\cos(\frac{\pi\alpha_\ell}{2})e^{-\frac{1}{2}v^2}dv\Big)^2.
\end{align*}
So that 
\begin{align*}
	&\E\big[L^{(\boldsymbol{\alpha})}_{+,\varepsilon}(T,0)\big]^2\\\ge&\E\Big[\E\big(L^{(\boldsymbol{\alpha})}_{+,\varepsilon}(T,0)|X^\ell,\ell\in I\big)\Big]^2\\=&\frac{2}{(2\pi)^{2d}}\int_{[0,T]^2_<}\prod_{\ell\in I}\bigg(\int_{\R^{2}}(-1)^{\alpha_\ell}v_1^{\alpha_\ell}v_2^{\alpha_\ell}\exp\Big(-\frac{1}{2}\big(Av_1^2+2Bv_1v_2+Cv_2^2\big)\Big)dv_1dv_2\bigg)\\&\qquad\qquad\qquad\qquad\qquad\qquad\times\prod_{\ell\notin I}(AC)^{-\frac{\alpha_\ell+1}{2}}\Big(\int_{\R}|v|^{\alpha_\ell}\cos(\frac{\pi\alpha_\ell}{2})e^{-\frac{1}{2}v^2}dv\Big)^2dtds\\\ge&c\int_{[0,T]^2_<}\prod_{\ell\in I}\frac{B^{\alpha_\ell}}{(AC-B^2)^{\frac{2\alpha_\ell+1}{2}}}\times\prod_{\ell\notin I}(AC)^{-\frac{\alpha_\ell+1}{2}}dtds\\\ge&c\int_{0}^{\widetilde{T}}\int_{s}^{\widetilde{s}(s)}\prod_{\ell\in I}\frac{B^{\alpha_\ell }}{(AC-B^2)^{\frac{2\alpha_\ell+1}{2}}}\times\prod_{\ell\notin I}(AC)^{-\frac{\alpha_\ell+1}{2}}dtds\\\ge&c\int_{0}^{\widetilde{T}}\int_{s}^{\widetilde{s}(s)}\prod_{\ell\in I}\frac{s^{2H\alpha_\ell }}{(s^{2H}+\varepsilon)^{2\alpha_\ell+1}}\times\prod_{\ell\notin I}(s^{2H}+\varepsilon)^{-\alpha_\ell-1}dtds\\=&c'\int_{0}^{\widetilde{T}}\prod_{\ell\in I}\frac{s^{2H\alpha_\ell }}{(s^{2H}+\varepsilon)^{2\alpha_\ell+1}}\times\prod_{\ell\notin I}(s^{2H}+\varepsilon)^{-\alpha_\ell-1}sds\\=&c'\int_{0}^{\widetilde{T}}\frac{s^{1+2H\sum_{\ell\in I}\alpha_\ell }}{(s^{2H}+\varepsilon)^{\sum_{\ell\in I}\alpha_\ell+|\boldsymbol{\alpha}|+d}}ds,
\end{align*}
where $\widetilde{T}=\Big(1+\sqrt[H]{\frac{2\kappa_H}{K_H}}\Big)^{-1}T$, $\widetilde{s}(s)=\Big(1+\sqrt[H]{\frac{2\kappa_H}{K_H}}\Big)s$ (see figure \ref{intarea1}), the first inequality comes from Jensen's inequality, we use \eqref{Ie} in the second inequality and the last inequality comes from Propositions \ref{acd} and \ref{b}. 

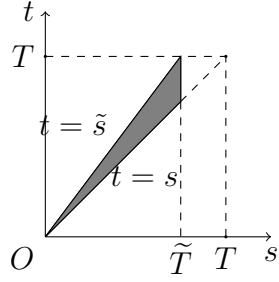
\begin{figure}[H]
	\centering
	\begin{tikzpicture}[scale=0.6]
		\draw[->](0,0)node[below left]{$O$}--(5,0)node[below]{$s$};
		\draw[->](0,0)--(0,5)node[left]{$t$};
		\draw[dashed](3,3)--(3,0);
		\draw[dashed](3,3)--(4,4);
		\draw[-](0,0)--(3,4);
		\draw[-](0,0)--(3,3);
		\draw[-](3,4)--(3,3);
		\draw(4,0)node[below]{$T$};
		\fill(4,0)circle(1pt);
		\fill(0,4)circle(1pt);
		\fill(4,4)circle(1pt);
		\draw(0,4)node[left]{$T$};
		\draw[dashed](0,4)--(4,4);
		\draw[dashed](4,0)--(4,4);
		\fill(3,0.1)node[below]{$\widetilde{T}$};
		\fill(1.6,2.5)node[left]{$t=\tilde{s}$};
		\fill(2.2,1.8)node[below]{$t=s$};
		\filldraw[fill=gray](0,0)--(3,3)--(3,4)--(0,0);
	\end{tikzpicture}
	\caption{The shadow stands for the area $\{(s,t):0<s<\widetilde{T},s<t<\widetilde{s}(s)\}$}
	\label{intarea1}
\end{figure} 

Then using Lemma \ref{lmaxiao}, because $Hd>1$, we get 
$\frac{2+2H\sum_{\ell\in I}\alpha_\ell}{\sum_{\ell\in I}\alpha_\ell+|\boldsymbol{\alpha}|+d}-2H<0$
and 
\[\E\big[L^{(\boldsymbol{\alpha})}_{+,\varepsilon}(T,0)\big]^2\ge c\varepsilon^{\frac{2+2H\sum_{\ell\in I}\alpha_\ell}{\sum_{\ell\in I}\alpha_\ell+|\boldsymbol{\alpha}|+d}-2H}\]
for any $\varepsilon\in(0,\frac12)$, which completes statements {\rm(2)} and {\rm(3)} in Theorem \ref{nes}.

{\bf Part III: The case that $\alpha_\ell$ is integer and $x_\ell\ne0$ for some $\ell\in\{1,2,\cdots,d\}$}

In this part, we assume $X\in\widetilde{G}^{d,H}_{S,U}$ and for some $\ell_0\in\{1,2,\cdots,d\}$, $\alpha_{\ell_0}\in\Z$ and $x_{\ell_0}\ne0$. Like {\bf Part I}, we define
$$\widetilde{I}_{+,\varepsilon}^{(\boldsymbol{\alpha})}(x)=\frac{2}{(2\pi)^{2d}}\int_{[0,T]^2_<}\prod_{\ell=1}^{d}\widetilde{J}_{+}(\alpha_\ell,x_\ell,t,s)dtds.$$
Then we will prove that there exists a positive function $f_{\boldsymbol{\alpha}}(\varepsilon)$ converging to 0 as $\varepsilon\downarrow0$ and a positive constant $c_{\boldsymbol{\alpha}}$, s.t. 
\[\lim_{\varepsilon\downarrow0}f_{\boldsymbol{\alpha}}(\varepsilon)\Big|\E\big[L^{(\boldsymbol{\alpha})}_{+,\varepsilon}(T,x)\big]^2-\widetilde{I}_{+,\varepsilon}^{(\boldsymbol{\alpha})}(x)\Big|=0\hbox{ and }\liminf_{\varepsilon\downarrow0}f_{\boldsymbol{\alpha}}(\varepsilon)\widetilde{I}_{+,\varepsilon}^{(\boldsymbol{\alpha})}(x)\ge c_{\boldsymbol{\alpha}}.\]

 Then we give the following proposition:

\begin{proposition}\label{a-b&c-b}
	For $X\in\widetilde{G}^{d,H}_{S,U}$, there exist a positive constant $\widetilde{K}$ and a non-negative decreasing function $\widetilde{\beta}(\gamma):(1,\infty)\to\mathbb{R}$ with $\lim\limits_{\gamma\to\infty}\widetilde{\beta}(\gamma)=0$, s.t.
	\begin{align}\label{btilde}
		M:=\max\{|A-B|,|C-B|\}\le \widetilde{K}\Big(\widetilde{\beta}(\gamma)\big(s^{2H}+\varepsilon\big)^{\frac12}\big((t-s)^{2H}+\varepsilon\big)^{\frac12}+\gamma^H\big((t-s)^{2H}+\varepsilon\big)\Big)
	\end{align}
	for all $t>s>0$, $\gamma>1$ and $\varepsilon>0$. 
\end{proposition}
\begin{proof}
	For $C-B$, we have 
	\begin{align*}
		|C-B|&=\big|\E \big[X^1_s(X^1_s-X^1_t)\big]+\varepsilon\big|\\&\le\big|\E \big[X^1_s(X^1_t-X^1_s)\big]\big|+\varepsilon\\&\le K\big(\widetilde{\beta}(\gamma)s^{H}(t-s)^{H}+\gamma^H(t-s)^{2H}\big)+\varepsilon\\&\le(K\vee1)\Big(\widetilde{\beta}(\gamma)\big(s^{2H}+\varepsilon\big)^{\frac12}\big((t-s)^{2H}+\varepsilon\big)^{\frac12}+\gamma^H\big((t-s)^{2H}+\varepsilon\big)\Big),
	\end{align*}
	where in the second inequality we use Lemma \ref{bocoi}. For $A-B$, we have 
	\begin{align*}
		|A-B|&=\big|\E \big[X^1_t(X^1_t-X^1_s)\big]+\varepsilon\big|\\&\le\E(X^1_t-X^1_s)^2+\big|\E \big[X^1_s(X^1_s-X^1_t)\big]\big|+\varepsilon\\&\le(K_H\vee1)\big((t-s)^{2H}+\varepsilon\big)+K\big(\widetilde{\beta}(\gamma)s^{H}(t-s)^{H}+\gamma^H(t-s)^{2H}\big)\\&\le((K_H\vee1)\vee K)\Big(\widetilde{\beta}(\gamma)\big(s^{2H}+\varepsilon\big)^{\frac12}\big((t-s)^{2H}+\varepsilon\big)^{\frac12}+\gamma^H\big((t-s)^{2H}+\varepsilon\big)\Big),
	\end{align*}
	where in the second inequality we use Lemma \ref{bocoi} and Proposition \ref{acd}. 
\end{proof}

For those $\ell\ne\ell_0$, we have 
\begin{align*}
	J_+(\alpha_\ell,x_\ell,t,s)=\frac{G}{D^{\frac{2\alpha_\ell+1}{2}}}\int_{\R^2}\big(-&\iota u_1-\iota\frac{A-B}{\sqrt{\Delta}} u_2\big)^{\alpha_\ell}_0\\\times\big(&\iota u_1+\iota \frac{C-B}{\sqrt{\Delta}} u_2\big)^{\alpha_\ell}_0e^{-\frac12u_1^2-\frac12u_2^2+\iota Gu_2x_\ell}du_1du_2
\end{align*}
and 
\begin{align*}
	\widetilde{J}_+(\alpha_\ell,x_\ell,t,s)=&\frac{G}{D^{\frac{2\alpha_\ell+1}{2}}}\int_{\R^2}\big(-\iota u_1\big)^{\alpha_\ell}_0\times\big(\iota u_1\big)^{\alpha_\ell}_0e^{-\frac12u_1^2-\frac12u_2^2+\iota Gu_2x_\ell}du_1du_2\\=&\frac{G}{D^{\frac{2\alpha_\ell+1}{2}}}e^{-\frac{G^2x_{\ell}^2}{2}}\int_{\R^2}|u_1|^{2\alpha_\ell}_0e^{-\frac12u_1^2-\frac12u_2^2}du_1du_2.
\end{align*}
Using Lemma \ref{diff}, we get 
\begin{align*}
	&\Big|\big(-\iota u_1-\iota\frac{A-B}{\sqrt{\Delta}} u_2\big)^{\alpha_\ell}_0\times\big(\iota u_1+\iota \frac{C-B}{\sqrt{\Delta}} u_2\big)^{\alpha_\ell}_0-\big(-\iota u_1\big)^{\alpha_\ell}_0\times\big(\iota u_1\big)^{\alpha_\ell}_0\Big|\\\le&\Big|\big(-\iota u_1-\iota\frac{A-B}{\sqrt{\Delta}} u_2\big)^{\alpha_\ell}_0-\big(-\iota u_1\big)^{\alpha_\ell}_0\Big|\cdot\Big|\big(\iota u_1+\iota \frac{C-B}{\sqrt{\Delta}} u_2\big)^{\alpha_\ell}_0\Big|\\+&\Big|\big(\iota u_1+\iota \frac{C-B}{\sqrt{\Delta}} u_2\big)^{\alpha_\ell}_0-\big(\iota u_1\big)^{\alpha_\ell}_0\Big|\cdot\big|\big(-\iota u_1\big)^{\alpha_\ell}_0\big|\\\le &c\Big(|u_1|^{\alpha_\ell-\alpha_\ell^*}\Big|\frac{A-B}{\sqrt{\Delta}} u_2\Big|^{\alpha_\ell^*}+\Big|\frac{A-B}{\sqrt{\Delta}} u_2\Big|^{\alpha_\ell}\Big)\Big(|u_1|^{\alpha_\ell}+\big|\frac{C-B}{\sqrt{\Delta}}u_2\big|^{\alpha_\ell}\Big)\\+&c\Big(|u_1|^{\alpha_\ell-\alpha_\ell^*}\Big|\frac{C-B}{\sqrt{\Delta}} u_2\Big|^{\alpha_\ell^*}+\Big|\frac{C-B}{\sqrt{\Delta}} u_2\Big|^{\alpha_\ell}\Big)|u_1|^{\alpha_\ell},
\end{align*}
so that 
\begin{align*}
	&|J_+(\alpha_{\ell_0},x_{\ell_0},t,s)-\widetilde{J}_+(\alpha_{\ell_0},x_{\ell_0},t,s)|\\\le&c\Bigg(\int_{\R^2}\Big(\big|\frac{A-B}{\sqrt{\Delta}}\big|^{\alpha_{\ell_0}^*}+\big|\frac{C-B}{\sqrt{\Delta}}\big|^{\alpha_{\ell_0}^*}\Big)|u_2|^{\alpha_{\ell_0}^*}\big|u_1\big|^{2\alpha_{\ell_0}-\alpha^*_{\ell_0}}e^{-\frac12u_1^2-\frac12u_2^2}du_1du_2
	\\&\quad+\int_{\R^2}\Big(\big|\frac{A-B}{\sqrt{\Delta}}\big|^{\alpha_{\ell_0}}+\big|\frac{C-B}{\sqrt{\Delta}}\big|^{\alpha_{\ell_0}}\Big)|u_2|^{\alpha_{\ell_0}}\big|u_1\big|^{\alpha_{\ell_0}}e^{-\frac12u_1^2-\frac12u_2^2}du_1du_2
	\\&\quad+\int_{\R^2}\big|\frac{A-B}{\sqrt{\Delta}}\big|^{\alpha_{\ell_0}^*}\big|\frac{C-B}{\sqrt{\Delta}}\big|^{\alpha_{\ell_0}}|u_2|^{\alpha_{\ell_0}+\alpha_{\ell_0}^*}\big|u_1\big|^{\alpha_{\ell_0}-\alpha_{\ell_0}^*}e^{-\frac12u_1^2-\frac12u_2^2}du_1du_2
	\\&\quad+\int_{\R^2}\big|\frac{A-B}{\sqrt{\Delta}}\big|^{\alpha_{\ell_0}}\big|\frac{C-B}{\sqrt{\Delta}}\big|^{\alpha_{\ell_0}}|u_2|^{2\alpha_{\ell_0}}e^{-\frac12u_1^2-\frac12u_2^2}du_1du_2\Bigg)\times\frac{G}{D^{\frac{2\alpha_{\ell_0}+1}{2}}}
	\\\le&c\Big(\Big|\frac{M}{\sqrt{\Delta}}\Big|^{\alpha_{\ell_0}^*}+\Big|\frac{M}{\sqrt{\Delta}}\Big|^{\alpha_{\ell_0}}+\Big|\frac{M}{\sqrt{\Delta}}\Big|^{\alpha_{\ell_0}+\alpha_{\ell_0}^*}+\Big|\frac{M}{\sqrt{\Delta}}\Big|^{2\alpha_{\ell_0}}\Big)\times\frac{G}{D^{\frac{2\alpha_{\ell_0}+1}{2}}},
\end{align*}
where letting $\zeta_\ell^{(1)}=\alpha^*_\ell,\zeta_\ell^{(2)}=\alpha_\ell,\zeta_\ell^{(3)}=\alpha_\ell+\alpha_\ell^*,\zeta_\ell^{(4)}=2\alpha_\ell$ and using Propositions \ref{acd}, \ref{dgc-b} and \ref{a-b&c-b}., we get 
\begin{equation}\label{difff}
	\begin{split}
		|J_+(\alpha_{\ell},x_{\ell},t,s)-\widetilde{J}_+(\alpha_{\ell},x_{\ell},t,s)|	&\le c\bigg(\beta_0(\gamma)\big((t-s)^{2H}+\varepsilon\big)^{-\frac{2\alpha_\ell+1}{2}}\big(s^{2H}+\varepsilon\big)^{-\frac{1}{2}}\\&+ \gamma^{2H\alpha_\ell}\sum_{k=1}^{4}\big((t-s)^{2H}+\varepsilon\big)^{-\frac{2\alpha_\ell-\zeta^{(k)}_\ell+1}{2}}\big(s^{2H}+\varepsilon\big)^{-\frac{1+\zeta^{(k)}_\ell}{2}}\bigg)
	\end{split}
\end{equation}
and 
\begin{equation}\label{difff''}
	\begin{split}
		|\widetilde{J}_+(\alpha_{\ell},x_{\ell},t,s)|	\le&c\big((t-s)^{2H}+\varepsilon\big)^{-\frac{2\alpha_{\ell_0}+1}{2}}\big(s^{2H}+\varepsilon\big)^{-\frac{1}{2}}
	\end{split}
\end{equation}
for some non-negative decreasing function ${\beta}_0(\gamma):(1,\infty)\to\mathbb{R}$ with $\lim\limits_{\gamma\to\infty}{\beta}_0(\gamma)=0$ and any $0<s<t$ and $\gamma>1$.

For $\ell_0$, because $\alpha_{\ell_0}$ is integer, using Lemma \ref{fouriert}, we get 
\begin{align*}
	J_+(\alpha_{\ell_0},x_{\ell_0},t,s)=\frac{G}{D^{\frac{2\alpha_{\ell_0}+1}{2}}}e^{-\frac{G^2x_{\ell_0}^2}{2}}\int_{\R^2}\big(-&\iota u_1+\iota\frac{A-B}{\sqrt{\Delta}} (u_2+\iota x_{\ell_0} G)\big)^{\alpha_{\ell_0}}\\\times\big(&\iota u_1+\iota \frac{C-B}{\sqrt{\Delta}} (u_2+\iota x_{\ell_0} G)\big)^{\alpha_{\ell_0}}e^{-\frac12u_1^2-\frac12u_2^2}du_1du_2
\end{align*}
and
\begin{align*}
	\widetilde{J}_+(\alpha_{\ell_0},x_{\ell_0},t,s)=\frac{G}{D^{\frac{2\alpha_{\ell_0}+1}{2}}}e^{-\frac{G^2x_{\ell_0}^2}{2}}\int_{\R^2}|u_1|^{2\alpha_{\ell_0}}e^{-\frac12u_1^2-\frac12u_2^2}du_1du_2.
\end{align*}
Using Lemma \ref{diff}, 
\begin{align*}
		&|J_+(\alpha_{\ell_0},x_{\ell_0},t,s)-\widetilde{J}_+(\alpha_{\ell_0},x_{\ell_0},t,s)|\\\le&c\Bigg(\int_{\R^2}\Big(\big|\frac{A-B}{\sqrt{\Delta}}\big|^{\alpha_{\ell_0}^*}+\big|\frac{C-B}{\sqrt{\Delta}}\big|^{\alpha_{\ell_0}^*}\Big)\big(|u_2|+|x_{{\ell_0}}||G|\big)^{\alpha_{\ell_0}^*}\big|u_1\big|^{2\alpha_{\ell_0}-\alpha^*_{\ell_0}}e^{-\frac12u_1^2-\frac12u_2^2}du_1du_2
	\\&\quad+\int_{\R^2}\Big(\big|\frac{A-B}{\sqrt{\Delta}}\big|^{\alpha_{\ell_0}}+\big|\frac{C-B}{\sqrt{\Delta}}\big|^{\alpha_{\ell_0}}\Big)\big(|u_2|+|x_{\ell_0}||G|\big)^{\alpha_{\ell_0}}\big|u_1\big|^{\alpha_{\ell_0}}e^{-\frac12u_1^2-\frac12u_2^2}du_1du_2
	\\&\quad+\int_{\R^2}\big|\frac{A-B}{\sqrt{\Delta}}\big|^{\alpha_{\ell_0}^*}\big|\frac{C-B}{\sqrt{\Delta}}\big|^{\alpha_{\ell_0}}\big(|u_2|+|x_{\ell_0}||G|\big)^{\alpha_{\ell_0}+\alpha_{\ell_0}^*}\big|u_1\big|^{\alpha_{\ell_0}-\alpha_{\ell_0}^*}e^{-\frac12u_1^2-\frac12u_2^2}du_1du_2
	\\&\quad+\int_{\R^2}\big|\frac{A-B}{\sqrt{\Delta}}\big|^{\alpha_{\ell_0}}\big|\frac{C-B}{\sqrt{\Delta}}\big|^{\alpha_{\ell_0}}\big(|u_2|+|x_{{\ell_0}}||G|\big)^{2\alpha_{\ell_0}}e^{-\frac12u_1^2-\frac12u_2^2}du_1du_2\Bigg)\times\frac{G}{D^{\frac{2\alpha_{\ell_0}+1}{2}}}e^{-\frac{G^2x_{\ell_0}^2}{2}}\\\le&c\Big(\Big|\frac{M}{\sqrt{\Delta}}\Big|^{\alpha_{\ell_0}^*}+\Big|\frac{M}{\sqrt{\Delta}}\Big|^{\alpha_{\ell_0}}+\Big|\frac{M}{\sqrt{\Delta}}\Big|^{\alpha_{\ell_0}+\alpha_{\ell_0}^*}+\Big|\frac{M}{\sqrt{\Delta}}\Big|^{2\alpha_{\ell_0}}\Big)\times\big(1+|x_{{\ell_0}}||G|\big)^{2\alpha_{\ell_0}}\times\frac{G}{D^{\frac{2\alpha_{\ell_0}+1}{2}}}e^{-\frac{G^2x_{\ell_0}^2}{2}},
\end{align*}
where using Proposition \ref{acd}, Proposition \ref{dgc-b}, Proposition \ref{a-b&c-b} and the fact that $0<\omega^{-\alpha}e^{-\omega^{-1}}\le\alpha^\alpha e^{-\alpha}$ for any $\alpha,\omega>0$, we get 
\begin{equation}\label{diffff}
	\begin{split}
			&|J_+(\alpha_{\ell_0},x_{\ell_0},t,s)-\widetilde{J}_+(\alpha_{\ell_0},x_{\ell_0},t,s)|	\\\le& c\bigg(\beta_0(\gamma)\big((t-s)^{2H}+\varepsilon\big)^{-\frac{2\alpha_\ell+1}{2}}+ \gamma^{2H\alpha_\ell}\sum_{k=1}^{4}\big((t-s)^{2H}+\varepsilon\big)^{-\frac{2\alpha_\ell-\zeta^{(k)}_\ell+1}{2}}\bigg)e^{-\frac{G^2x_{\ell_0}^2}{4}}\\\le&c\bigg(\beta_0(\gamma)\big((t-s)^{2H}+\varepsilon\big)^{-\frac{2\alpha_\ell+1}{2}}+ \gamma^{2H\alpha_\ell}\sum_{k=1}^{4}\big((t-s)^{2H}+\varepsilon\big)^{-\frac{2\alpha_\ell-\zeta^{(k)}_\ell+1}{2}}\bigg)\\&\qquad\qquad\qquad\qquad\qquad\qquad\qquad\qquad\times\exp\Big(-\frac{x_{\ell_0}^2}{4}\frac{\kappa_H\wedge2}{(K_H\vee2)^2}\cdot \big(s^{2H}+\varepsilon\big)^{-1}\Big)
	\end{split}
\end{equation}
 and 
\begin{equation}\label{diffff''}
	\begin{split}
		|\widetilde{J}_+(\alpha_{\ell_0},x_{\ell_0},t,s)|	\le&c\bigg(\beta_0(\gamma)\big((t-s)^{2H}+\varepsilon\big)^{-\frac{2\alpha_{\ell_0}+1}{2}}\bigg)\times\exp\Big(-\frac{x_{\ell_0}^2}{4}\frac{\kappa_H\wedge2}{(K_H\vee2)^2}\cdot \big(s^{2H}+\varepsilon\big)^{-1}\Big)
	\end{split}
\end{equation}
for  $\zeta_\ell^{(1)}=\alpha^*_\ell,\zeta_\ell^{(2)}=\alpha_\ell,\zeta_\ell^{(3)}=\alpha_\ell+\alpha_\ell^*,\zeta_\ell^{(4)}=2\alpha_\ell$, some non-negative decreasing function ${\beta}_0(\gamma):(1,\infty)\to\mathbb{R}$ with $\lim\limits_{\gamma\to\infty}{\beta}_0(\gamma)=0$ and any $0<s<t$ and $\gamma>1$. 

Then like the arguments in \eqref{thm211}, we get
\begin{align}\label{thm231}
		&\big|\E\big[L^{(\boldsymbol{\alpha})}_{+,\varepsilon}(T,x)\big]^2-\widetilde{I}_{+,\varepsilon}^{(\boldsymbol{\alpha})}(x)\big|\nonumber\\\le&c\int_{[0,T]^2_<}\Big|\prod_{\ell=1}^{d}J_+(\alpha_\ell,x_\ell,t,s)-\prod_{\ell=1}^{d}\widetilde{J}_+(\alpha_\ell,x_\ell,t,s)\Big|dtds\nonumber\\\le&c\int_{[0,T]^2_<}\Big|J_+(\alpha_{\ell_0},x_{\ell_0},t,s)-\widetilde{J}_+(\alpha_{\ell_0},x_{\ell_0},t,s)\Big|\cdot\Big|\prod_{\ell=1:\ell\ne\ell_0}^{d}J_+(\alpha_\ell,x_\ell,t,s)\Big|dtds\nonumber\\+&c\int_{[0,T]^2_<}\Big|\widetilde{J}_+(\alpha_{\ell_0},x_{\ell_0},t,s)\Big|\cdot\Big|\prod_{\ell=1:\ell\ne\ell_0}^{d}J_+(\alpha_\ell,x_\ell,t,s)-\prod_{\ell=1:\ell\ne\ell_0}^{d}\widetilde{J}_+(\alpha_\ell,x_\ell,t,s)\Big|dtds\nonumber\\\le&\left\{\begin{array}{cl}
			c\varepsilon^{\frac{1}{2H}-|\boldsymbol{\alpha}|-\frac{d}{2}+\beta}+c[\beta_0(\gamma)]^{d}\varepsilon^{\frac{1}{2H}-|\boldsymbol{\alpha}|-\frac{d}{2}},&H(2|\boldsymbol{\alpha}|+d)>1,\\c+c[\beta_0(\gamma)]^{d}\ln(1+\varepsilon^{-\frac12}), &H(2|\boldsymbol{\alpha}|+d)=1,
		\end{array}\right.
\end{align}
for any $\varepsilon\in(0,\frac12)$ and $\gamma>1$, where we use \eqref{difff}, \eqref{difff''}, \eqref{diffff} and \eqref{diffff''} in the last inequality and 
\[\beta=\frac{1}{4}\Big\{\min\{\alpha_\ell^{*}:\alpha_\ell^{*}>0,\ell=1,\cdots,d\}\wedge\big(\frac{1}{H}-d\big)\wedge\big(2|\boldsymbol{\alpha}|+d-\frac{1}{H}\big)\Big\}.\]

Moreover, using Proposition \ref{acd}, Proposition \ref{dgc-b} and coordinate transform $u=t-s,v=s$,
\begin{align}\label{thm222}
	\widetilde{I}_{+,\varepsilon}^{(\boldsymbol{\alpha})}(x)&=\frac{2}{(2\pi)^{2d}}\int_{[0,T]^2_<}\prod_{\ell=1}^{d}\widetilde{J}_{+}(\alpha_\ell,x_\ell,t,s)dtds\nonumber\\&\ge c\int_{[0,T]^2_<}\frac{G^d}{D^{\frac{2|\boldsymbol{\alpha}|+d}{2}}}e^{-\frac{G^2|x|^2}{2}}dtds\nonumber\\&\ge c\int_{\frac{T}{4}}^{\frac{T}{2}}\big(u^{2H}+\varepsilon\big)^{-\frac{d}{2}}\exp\Big(-\frac{\kappa_H\wedge2}{2(K_H\vee2)^2}\cdot\frac{|x|^2}{u^{2H}+\varepsilon}\Big)du\times\int_{0}^{\frac{T}{2}}\big(v^{2H}+\varepsilon\big)^{-\frac{2|\boldsymbol{\alpha}|+d}{2}}dv\nonumber\\&\ge \left\{\begin{array}{cl}
		c\varepsilon^{\frac{1}{2H}-|\boldsymbol{\alpha}|-\frac{d}{2}},&H(2|\boldsymbol{\alpha}|+d)>1,\\c\ln(1+\varepsilon^{-\frac12}),&H(2|\boldsymbol{\alpha}|+d)=1,
	\end{array}\right.
\end{align}
for any $\varepsilon\in(0,\frac12)$. So combining \eqref{thm231} with \eqref{thm222}, we get 
\begin{align*}
	\liminf_{\varepsilon\downarrow0}\E\big[L^{(\boldsymbol{\alpha})}_{+,\varepsilon}(T,x)\big]^2=+\infty,
\end{align*}
which completes the proof.

\section{Technical Lemmas}
In this section we give some lemmas based on contour integration, which will be used in this article.
\begin{lemma}\label{Ff}
	For any $0<\alpha<1$, $u\in\mathbb{R}$ and $\mathrm{sgn}(u)=\left\{\begin{array}{cl}
		1,&u>0;\\0,&u=0;\\-1,&u<0. 
	\end{array}\right.$, we have $$\frac{\alpha}{\Gamma(1-\alpha)}\int_0^{+\infty}\frac{1}{v^{1+\alpha}}(1-e^{-\iota uv})dv=|{u}|^\alpha{e}^{\iota{\frac{\pi\alpha}{2}\mathrm{sgn}(u)}}.$$
\end{lemma}
\begin{proof}
	Given $0<\alpha<1$ and $u\in\R$, we have 
	\begin{align*}
			&\frac{\alpha}{\Gamma(1-\alpha)}\int_0^{+\infty}\frac{1}{v^{1+\alpha}}(1-e^{-\iota uv})dv\\=&\frac{\alpha}{\Gamma(1-\alpha)}\Big(\int_0^{+\infty}\frac{1-\cos(uv)}{v^{1+\alpha}}dv+\iota\int_0^{+\infty}\frac{\sin(uv)}{v^{1+\alpha}}dv\Big)\\=&\frac{\alpha}{\Gamma(1-\alpha)}\Big(|u|^\alpha\int_0^{+\infty}\frac{1-\cos(v)}{v^{1+\alpha}}dv+\iota{u}|u|^{\alpha-1}\int_0^{+\infty}\frac{\sin(v)}{v^{1+\alpha}}dv\Big)\\=&\frac{\alpha}{\Gamma(1-\alpha)}\Big(|u|^\alpha\mathrm{Re}\int_0^{+\infty}\frac{1}{v^{1+\alpha}}(1-e^{-\iota v})dv+\iota{u}|u|^{\alpha-1}\mathrm{Im}\int_0^{+\infty}\frac{1}{v^{1+\alpha}}(1-e^{-\iota v})dv\Big),
	\end{align*}	
	where $\mathrm{Re}\,z$ and $\mathrm{Im}\, z$ are the real and imaginary part of a complex number $z$, respectively. So we only need to verify that 
	\[\int_0^{+\infty}\frac{1}{v^{1+\alpha}}(1-e^{-\iota v})dv=\frac{e^{\iota\frac{\pi\alpha}{2}}\Gamma(1-\alpha)}{\alpha}.\]
	To solve the problem, we introduce complex function $(\cdot)_0^{p}$ defined by 
	$$(z)_0^{p}=|z|^{p}e^{\iota{p}\arg{(z)}}.$$
	It is widely known that when $p$ is a non-negative integer $(z)_0^{p}=z^p$ is analytic on $\C$ and in other cases $(z)_0^{p}$ is analytic on $\mathbb{C}\backslash\{{y}:y\le0\}$.
\begin{figure}[H]
	\centering
	\begin{tikzpicture}[scale=0.7,thick]
		\draw[->](0,0)node[below left]{$O$}--(0.5,0)node[below]{$r$}--(3.2,0)node[below]{$R$}--(4.2,0)node[below]{$x$};
		\draw[->](0,0)--(0,0.5)node[left]{$r$}--(0,3.2)node[left]{$R$}--(0,4.2)node[right]{$y$};
		\fill(0.5,0)circle(1pt);
		\fill(0,0.5)circle(1pt);
		\fill(3.2,0)circle(1pt);
		\fill(0,3.2)circle(1pt);
		\draw(0.5,0)arc(0:90:0.5);
		\draw[->](0.5,0)arc(0:45:0.5)node[right]{$\gamma_{r}$};
		\draw(3.2,0)arc(0:90:3.2);
		\draw[-<](3.2,0)arc(0:45:3.2)node[right]{$\gamma_R$};
		\draw[->](3.2,0)--(1.6,0)node[below]{$\gamma_{1}$};
		\draw[->](0,0)--(0,1.6)node[right]{$\gamma_{2}$};
	\end{tikzpicture}
	\caption{The contour of $(z)_0^{-\alpha-1}(1-e^{-z})$}
	\label{contour1}
\end{figure}
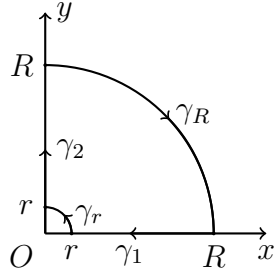 
	Now assume $0<r<R$, by Cauchy-Goursat Theorem for analytic function we have \[\int_{\gamma_{1}+\gamma_r+\gamma_{2}+\gamma_R}(z)_0^{-\alpha-1}(1-e^{-z})dz=0,\]
	where the four curves $\gamma_{1},\gamma_r,\gamma_{2},\gamma_R$(see Figure \ref{contour1}) are  
	\begin{align*}
		\gamma_1:\gamma_1(t)&=-t,t\in[-R,-r];\qquad\ \ \ 
		\gamma_{2}:\gamma_{2}(t)=\iota{t},t\in[r,R];\\
		\gamma_r:\gamma_r(t)&=re^{\iota{t}},t\in[0,\frac{\pi}{2}];\qquad\qquad
		\gamma_R:\gamma_R(t)=Re^{-\iota{t}},t\in[-\frac{\pi}{2},0],
	\end{align*}  
Here we have 	\[\lim\limits_{r\to0}\int_{\gamma_r}(z)_0^{-\alpha-1}(1-e^{-z})dz=\lim\limits_{R\to+\infty}\int_{\gamma_R}(z)_0^{-\alpha-1}(1-e^{-z})dz=0\]
	by the fact 
	\begin{equation*}
		\begin{split}
			\Big|\int_{\gamma_R}(z)_0^{-\alpha-1}(1-e^{-z})dz\Big|\le\int_{0}^{\frac{\pi}{2}}R^{-\alpha}|1-e^{-Re^{\iota{t}}}|dt\le\int_{0}^{\frac{\pi}{2}}R^{-\alpha}(1+e^{-R\cos(t)})dt\le{\pi}R^{-\alpha}
		\end{split}
	\end{equation*}
	and 
	\begin{equation*}
		\begin{split}
			\Big|\int_{\gamma_r}(z)_0^{-\alpha-1}(1-e^{-z})dz\Big|\le\int_{0}^{\frac{\pi}{2}}r^{-\alpha}|1-e^{-re^{\iota{t}}}|dt=r^{1-\alpha}\int_{0}^{\frac{\pi}{2}}\Big|\frac{1-e^{-re^{\iota{t}}}}{re^{\iota{t}}}\Big|dt\le\frac{\pi}{2}{r^{1-\alpha}}e^{r}.
		\end{split}
	\end{equation*}
	So that we have 
	\begin{equation*}
		\begin{split}
			\int_0^{+\infty}\frac{1}{v^{1+\alpha}}(1-e^{-\iota v})dv&=e^{\iota\frac{\pi\alpha}{2}}\lim\limits_{r\to0,R\to+\infty}\int_{\gamma_2}(z)_0^{-\alpha-1}(1-e^{-z})dz\\&=-e^{\iota\frac{\pi\alpha}{2}}\lim\limits_{r\to0,R\to+\infty}\int_{\gamma_1}(z)_0^{-\alpha-1}(1-e^{-z})dz\\&=e^{\iota\frac{\pi\alpha}{2}}\int_0^{+\infty}\frac{1}{v^{1+\alpha}}(1-e^{-v})dv=e^{\iota\frac{\pi\alpha}{2}}\frac{\Gamma(1-\alpha)}{\alpha},
		\end{split}
	\end{equation*}
	which completes the proof.
\end{proof}
\begin{lemma}\label{diff}
	For any $\alpha\ge0$ and complex function $(\iota\cdot)_0^{\alpha}$ defined in \eqref{alphapower}, there exists a constant $c_{\alpha,2}>0$, s.t. for any $y_1,y_2\in\mathbb{R}$, 
	\begin{equation}\label{mb}
		\big|(\iota y_1+\iota y_2)_0^{\alpha}-(\iota y_2)_0^{\alpha}\big|\le {c}_{\alpha,2}\big(|y_2|^{\bar{\alpha}-(\bar{\alpha}\wedge1)\mathbf{1}_{\{\widetilde{\alpha}=0\}}}|y_1|^{\widetilde{\alpha}+(\bar{\alpha}\wedge1)\mathbf{1}_{\{\widetilde{\alpha}=0\}}}+|y_1|^\alpha\big), 
	\end{equation} where $\bar{\alpha}$ is the largest integer no greater than $\alpha$ and $\widetilde{\alpha}=\alpha-\bar{\alpha}$ and $\bar{\alpha}\wedge1=\min\{\bar{\alpha},1\}$. Moreover, when $\alpha$ is integer, we get for some $c_{\alpha,3}>0$, 
	\begin{equation}\label{mb'}
		\big|(z_1+z_2)^{\alpha}-z_2^{\alpha}\big|\le {c}_{\alpha,3}\big(|z_2|^{{\alpha}-{\alpha}\wedge1}|z_1|^{{\alpha}\wedge1}+|z_1|^\alpha\big)
	\end{equation}
	for any $z_1,z_2\in\C$.
\end{lemma}
\begin{proof} 
	When $\alpha=0$, the inequality \eqref{mb} is obvious. When $0<\alpha<1$, we have $\bar{\alpha}=0$ and $\alpha=\widetilde{\alpha}$ so the inequality \eqref{mb} is from 
	\begin{equation}\label{in1}
		\begin{split}
			\big|(\iota y_1+\iota y_2)_0^{\alpha}-(\iota y_2)_0^{\alpha}\big|&=\Big|\frac{\alpha}{\Gamma(1-\alpha)}\int_0^{+\infty}\frac{1}{v^{1+\alpha}}(e^{-\iota y_2v}-e^{-\iota (y_1+y_2)v})dv\Big|\\&\le\frac{\alpha}{\Gamma(1-\alpha)}\int_0^{+\infty}\frac{1}{v^{1+\alpha}}|1-e^{-\iota y_1v}|dv\\&=\frac{\alpha2^{1-\alpha}}{\Gamma(1-\alpha)}\int_0^{+\infty}\frac{|\sin v|}{v^{1+\alpha}}dv\times|y_1|^\alpha,
		\end{split} 
	\end{equation}
	where we use Lemma \ref{Ff} in the first equality. 
	
	Now we turn to the case $\alpha\ge1$. If $\alpha$ is an integer we have $\widetilde{\alpha}=0$ and $\bar{\alpha}=\alpha$. At this time 
	\begin{equation}\label{in2}
		\begin{split}
			\big|(\iota y_1+\iota y_2)_0^{\alpha}-(\iota y_2)_0^{\alpha}\big|&=\big|(y_1+y_2)^{\alpha}-y_2^{\alpha}\big|=\big|\alpha(\theta y_1+y_2)^{\alpha-1}y_1\big|\\&\le\alpha\big(|y_1|+|y_2|\big)^{\alpha-1}|y_1|=\alpha2^{\alpha-1}(|y_2|^{\alpha-1}|y_1|+|y_1|^\alpha),
		\end{split}
	\end{equation}	
	where $\theta\in(0,1)$. If $\alpha$ is not an integer, we have $0<\widetilde{\alpha}<1$ and $\bar{\alpha}\ge1$ so that for some $c>0$ and $c'>0$, there exists 
	\begin{equation}
		\begin{split}
			&\big|(\iota y_1+\iota y_2)_0^{\alpha}-(\iota y_2)_0^{\alpha}\big|
			\\\le&|(\iota y_1+\iota y_2)_0^{\bar{\alpha}}|\cdot\big|(\iota y_1+\iota y_2)_0^{\widetilde{\alpha}}-(\iota y_2)_0^{\widetilde{\alpha}}\big|+|(\iota y_2)_0^{\widetilde{\alpha}}|\cdot\big|(\iota y_1+\iota y_2)_0^{\bar{\alpha}}-(\iota y_2)_0^{\bar{\alpha}}\big|\\\le& c\Big(|y_1+y_2|^{\bar{\alpha}}|y_1|^{\widetilde{\alpha}}+|y_2|^{\widetilde{\alpha}}\big(|y_2|^{\bar{\alpha}-1}|y_1|+|y_1|^{\bar{\alpha}}\big)\Big)\le c'\big(|y_2|^{\bar{\alpha}}|y_1|^{\widetilde{\alpha}}+|y_1|^\alpha\big)
		\end{split}
	\end{equation}
	for all $y_1,y_2\in\R$, where in the second inequality we use \eqref{in1} and \eqref{in2}, and in the last inequality we use the fact $|y_1+y_2|^{\bar{\alpha}}\le\big(|y_1|+|y_2|\big)^{\bar{\alpha}}$, $\max\{|y_1|,|y_2|\}\le\big(|y_1|+|y_2|\big)$ and  $$\big(|y_1|+|y_2|\big)^{\bar{\alpha}}\le2^{\bar{\alpha}}\big(|y_1|^{\bar{\alpha}}+|y_2|^{\bar{\alpha}}\big),$$ 
	which completes the proof of \eqref{mb}. The proof of \eqref{mb'} comes from the decomposition 
	\[\big|(z_1+z_2)^{\alpha}-z_1^\alpha\big|=\Big|\sum_{k=1}^{\alpha}\frac{\alpha!}{(\alpha-k)!k!}z_1^kz_2^{\alpha-k}\Big|\le2^{\alpha}|z_1|(|z_1|+|z_2|)^{\alpha-1}\le2^{2\alpha-1}\big(|z_2|^{\alpha-1}|z_1|+|z_1|^\alpha\big)\]
	for any $z_1,z_2\in \C$ and integer $\alpha\ge1$.
\end{proof}
\begin{lemma}\label{fouriert}
	For any real number $a>0$, $x\in\R$ and integer $m\ge0$, there exists 
	\begin{equation*}
		\begin{split}
			\int_{-\infty}^{+\infty}t^{m}e^{-\frac{1}{2}at^2}e^{\iota x t}dt=\int_{-\infty}^{+\infty}(t+\iota \frac{x}{a})^{m}e^{-\frac{1}{2}at^2}dt\times{e^{-\frac{x^2}{2a}}}.
		\end{split}
	\end{equation*}
\end{lemma}
\begin{proof}	First, when $x=0$, the equality is obvious. Then for $x>0$ and $x<0$, we consider complex function $F(z)=(z+\iota \frac{x}{a})^me^{-\frac{1}{2}az^2}$. For any $R>0$, denote curves(see Figure \ref{contour2}) 
	\begin{equation*}
		\begin{split}
			\gamma_1:\gamma_1(t)&=-R+\iota\sgn(x)(t-\frac{|x|}{a}),t\in[0,\frac{|x|}{a}],\quad\,\gamma_{2}:\gamma_2(t)=t,t\in[-R,R],\\
			\gamma_3:\gamma_3(t)&=R+\iota\sgn(x)(t-\frac{|x|}{a}),t\in[0,\frac{|x|}{a}],\qquad\gamma_{4}:\gamma_4(t)=t-\iota\frac{x}{a},t\in[-R,R].
		\end{split}
	\end{equation*}
	\begin{figure}[H]
		\centering
		\begin{tikzpicture}[thick]
			\draw[->](0,0)--(4.2,0)node[below]{$x$};
			\draw(0.5,0)circle(1pt)node[above]{$-R$};
			\draw(3.5,0)circle(1pt)node[above]{$R$};
			\draw(0.5,0)--(0.5,-1.5);
			\draw(3.5,0)--(3.5,-1.5);
			\draw(0.5,-1.5)--(3.5,-1.5);
			\draw[->](0.5,-1.5)--(0.5,-0.75)node[right]{$\gamma_1$};
			\draw[->](3.5,-1.5)--(3.5,-0.75)node[right]{$\gamma_3$};
			\draw[->](0.5,0)--(2,0)node[below]{$\gamma_{2}$};
			\draw[->](0.5,-1.5)--(2,-1.5)node[below]{$\gamma_4$};
		\end{tikzpicture}
		\begin{tikzpicture}[thick]
			\draw[->](0,0)--(4.2,0)node[below]{$x$};
			\draw(0.5,0)circle(1pt)node[below]{$-R$};
			\draw(3.5,0)circle(1pt)node[below]{$R$};
			\draw(0.5,0)--(0.5,1.5);
			\draw(3.5,0)--(3.5,1.5);
			\draw(0.5,1.5)--(3.5,1.5);
			\draw[->](0.5,1.5)--(0.5,0.75)node[right]{$\gamma_1$};
			\draw[->](3.5,1.5)--(3.5,0.75)node[right]{$\gamma_3$};
			\draw[->](0.5,0)--(2,0)node[below]{$\gamma_2$};
			\draw[->](0.5,1.5)--(2,1.5)node[below]{$\gamma_4$};
		\end{tikzpicture}
		\caption{The contour of $(z+\iota \frac{x}{a})^me^{-\frac{1}{2}az^2}$ when $x>0$ (left) and $x<0$ (right)}
		\label{contour2}
	\end{figure}
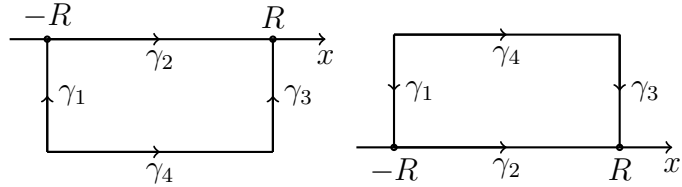 
	Because $F(z)$ is analytic on $\mathbb{C}$, we have 
	\[\int_{\gamma_1}F(z)dz+\int_{\gamma_2}F(z)dz=\int_{\gamma_4}F(z)dz+\int_{\gamma_3}F(z)dz,\]
	where 
	\begin{align*}
		\int_{\gamma_1}F(z)dz&=\iota\sgn(x)\int_0^{\frac{|x|}{a}}\big(-R+\iota\sgn(x)t\big)^m\exp\Big(-\frac{1}{2}a\big(-R+\iota\sgn(x)(t-\frac{|x|}{a})\big)^2\Big)dt,\\\int_{\gamma_3}F(z)dz&=\iota\sgn(x)\int_0^{\frac{|x|}{a}}\big(R+\iota\sgn(x)t\big)^m\exp\Big(-\frac{1}{2}a\big(R+\iota\sgn(x)(t-\frac{|x|}{a})\big)^2\Big)dt,\\\int_{\gamma_2}F(z)dz&=\int_{-R}^{R}(t+\iota \frac{x}{a})^me^{-\frac{1}{2}at^2}dt,\int_{\gamma_4}F(z)dz={e^{\frac{x^2}{2a}}}\int_{-R}^{R}t^{m}e^{-\frac{1}{2}at^2}e^{\iota x t}dt.
	\end{align*} 
	Let $R\to+\infty$, because for any $y\in\R$
	\begin{equation*}
		\begin{split}
			\sup_{0\le{t}\le{\frac{|x|}{a}}}\Big|\big(y+\iota\sgn(x)t\big)^m\exp\Big(-\frac{1}{2}a\big(y+\iota\sgn(x)(t-\frac{|x|}{a})\big)^2\Big)\Big|\le|y+\frac{|x|}{a}|^{m}e^{-\frac12ay^2+\frac12\frac{|x|^2}{a}},
		\end{split}
	\end{equation*}
	we have $\displaystyle\lim\limits_{R\to+\infty}\int_{\gamma_1}F(z)dz=\lim\limits_{R\to+\infty}\int_{\gamma_3}F(z)dz=0$ by dominated convergence theorem, which completes the proof.
\end{proof}
\section{Appendix}
In this part we will give a sufficient condition for the self-similar Gaussian process to possess the strong local nondeterminism property (\ref{slndp}) like \cite{hlx2023}. Assume that $Z=\{Z_t,\, t\geq 0\}$ is a one-dimensional centered Gaussian process satisfying the self-similarity with index $H\in (0,1)$ and $Y=\{Y_t=e^{-Ht}Z_{e^t},\, t\in\mathbb{R}\}$. Clearly, $Y$ is a one-dimensional centered Gaussian stationary process with covariance function $r(t)=\mathbb{E}[Y_0Y_t]=e^{-Ht}\mathbb{E}[Z_1Z_{e^t}]$ for any $t\in\mathbb{R}$, where $r(\cdot)$ is an even function on $\mathbb{R}$.
\begin{theorem} \label{glndp} Suppose that $r(t)$ is a non-negative decreasing function on $[0,\infty)$, $r(\cdot)\in L^1(\mathbb{R})$ and there exists a positive constant $c$ such that 
	\[
	\int^{\infty}_0 r(t)\cos(t\lambda) dt\geq \frac{c}{(|\lambda|+1)^{1+2H}}
	\]
	for all $\lambda\in\mathbb{R}$. Then there exists a positive constant $\kappa_{H}$ such that for any integer $m\ge 1$, any times $0=s_0<s_1,\dots,s_m,t<\infty$,  
	\begin{align} \label{slndpb}
		\Var\Big(Z_t| Z_{s_1},\dots, Z_{s_m}\Big)\geq\kappa_H \min\{|t-s_1|^{2H},\cdots,|t-s_m|^{2H},t^{2H}\}.
	\end{align} 
\end{theorem}
\begin{proof}
	By assumption, the spectral density $f$ of $Y$ has the expression 
	\begin{align*}  
		f(\lambda)=\frac{1}{\pi}\int^{\infty}_0 r(t)\cos(t\lambda) dt.
	\end{align*}
	Without loss of generality we can assume that $t\notin\{s_1,s_2,\cdots,s_n\}$. For the case $t<\min\limits_{1\le j\le m}s_j$, denote $$r=\min\{s_1-t,\cdots,s_m-t,t\}.$$ Then by the self-similarity of $Z$,
	\begin{align*}
		\Var\Big(Z_t| Z_{s_1},\dots, Z_{s_m}\Big)&\ge\Var\Big(Z_t| Z_{s}, s\ge t+r\Big)\\
		&=t^{2H}\Var\Big(Z_1| Z_{s}, s\ge1+r'\Big)\\
		&=t^{2H}\Var\Big(Y_0| Y_{\ln s}, s\ge1+r'\Big)\\
		&\geq t^{2H}\Var\Big(Y_0| Y_u, |u|\geq \frac{r'}{2}\Big),
	\end{align*}
	where $r'=r/t\in(0,1]$ and in the last inequality we use $\ln s\ge\ln (1+r')>\frac{r'}{2}$ for $r'\in (0,1]$. For the case $t>\min\limits_{1\le j\le m}s_j$, denote 
	$$r=\min\{|s_1-t|,\cdots,|s_m-t|,t\}.$$ Similarly, 	\begin{align*}
		\Var\Big(Z_t| Z_{s_1},\dots, Z_{s_m}\Big)&\ge\Var\Big(Z_t| Z_{s}, s\ge t+r\text{ or }0<s\le t-r\Big)\\
		&=t^{2H}\Var\Big(Z_1| Z_{s}, s\ge1+r'\text{ or }0<s\le1-r'\Big)\\
		&=t^{2H}\Var\Big(Y_0| Y_{\ln s}, s\ge1+r'\text{ or }0<s\le1-r'\Big)\\
		&\geq t^{2H}\Var\Big(Y_0| Y_u, |u|\geq \frac{r'}{2}\Big),
	\end{align*}
	where $r'=r/t\in(0,1)$ and in the last inequality we use $\ln (1+r')>\frac{r'}{2}$ and $\ln(1-r')<-\frac{r'}{2}$ for $r'\in (0,1)$. The rest of the proof can be obtained just using similar arguments in the proof of Lemma 1 in \cite{cd1982}.
\end{proof}

The following two corollaries show that bi-fractional Brownian motion $B^{H_0, K_0}$ and sub-fractional Brownian motion $S^{H}$ satisfy the strong local nondeterminism property (\ref{slndp}).

\begin{corollary} \label{bfbm}
	Assume that $B^{H_0, K_0}$ is the one-dimensional bi-fractional Brownian motion. Then there exists a positive constant $\kappa_{H_0, K_0}$ such that for any integer $m\ge 1$, any times $0=s_0<s_1,\dots,s_m,t<\infty$,  
	\begin{align*}  
		\Var\Big(B^{H_0, K_0}_t| B^{H_0, K_0}_{s_1},\dots, B^{H_0, K_0}_{s_m}\Big)\geq \kappa_{H_0, K_0} \min\{|t-s_1|^{2H_0K_0},\cdots,|t-s_m|^{2H_0K_0},t^{2H_0K_0}\}. 
	\end{align*} 
\end{corollary}
\begin{proof}
	Using Theorem \ref{glndp} and the arguments in the proof of Corollary 5.2 in \cite{hlx2023} gives the desired result.
\end{proof}
\begin{corollary} \label{sfbm}
	Assume that $S^H$ is the one-dimensional sub-fractional Brownian motion. Then there exists a positive constant $\kappa_{H}$ such that for any integer $m\ge 1$, any times $0=s_0<s_1,\dots,s_m,t<\infty$,  
	\begin{align*}  
		\Var\Big(S^H_t| S^H_{s_1},\dots, S^H_{s_m}\Big)\geq \kappa_{H}  \min\{|t-s_1|^{2H},\cdots,|t-s_m|^{2H},t^{2H}\}. 
	\end{align*} 
\end{corollary}
\begin{proof}
	Using Theorem \ref{glndp} and the arguments in the proof of Corollary 5.3 in \cite{hlx2023} gives the desired result.
\end{proof}

	$\begin{array}{cc}
		\begin{minipage}[t]{1\textwidth}
			{\bf Minhao Hong}\\
			School of Science, Shanghai Maritime University, Shanghai 201306, China\\
			\texttt{hongmhmath@163.com}
		\end{minipage}
		\hfill
	\end{array}$
	
	\bigskip
	$\begin{array}{cc}
		\begin{minipage}[t]{1\textwidth}
			{\bf Qian Yu}\\
			Department of Mathematics, Nanjing University of Aeronautics and Astronautics, Nanjing
			211106, China\\
			\texttt{qyumath@163.com}
		\end{minipage}
		\hfill
	\end{array}$

\begin{thebibliography}{99}
		\bibitem{oo2009}Ait Ouahra, M.  and Ouali, M.: Occupation time problems for fractional Brownian motion and some other self-similar processes. \textit{Random Operators and Stochastic Equations} {\bf 17}(1), 69--89, 2009.
		\bibitem{b1970}Berman, S. M.: Local times and sample function properties of stationary Gaussian processes, \textit{The Annals of Mathematical Statistics} {\bf 41}(4), 1260-1272, 1970.
		\bibitem{bgt2004} Bojdeckia, T., Gorostizab, L. and Talarczyk, A.:  Sub-fractional Brownian motion and its relation to occupation times. \textit{Statistics $\&$ Probability Letters} {\bf 69}, 405--419, 2004.
		\bibitem{cd1982} Cuzick, J. and Du Preez, J.P.: Joint continuity of Gaussian local times. \textit{Annals of Probability} {\bf 10}, 810--817, 1982.
		\bibitem{yg1992} Fitzsimmons, P. and Getoor, R.: Limit theorems and variation properties for fractional derivatives of the local time of a stable process. \textit{Annales de l'Institut Henri Poincare (B) Probability and Statistics} {\bf 28}, 311--333, 1992.
		\bibitem{ghx2019}Guo, J., Hu, Y. and Xiao, Y.: Higher-Order Derivative of Intersection Local Time for Two Independent Fractional Brownian Motions. \textit{Journal of Theoretical Probability} {\bf 32}, 1190--1201, 2019.
		\bibitem{hx2020}Hong, M. and Xu, F.: Derivatives of local times for some Gaussian fields. \textit{Journal of Mathematical Analysis and Applications} {\bf 484}, 123716, 2020.
		\bibitem{hx2021}Hong, M. and Xu, F.: Derivatives of local times for some Gaussian fields II. \textit{Statistics \& Probability Letters} {\bf 172}, 109063, 2021.	
		\bibitem{hlx2023}Hong, M., Liu, H. and Xu, F.: Limit theorems for additive functionals of some self-similar Gaussian processes, arXiv:2305.13146, 2023.
		
		\bibitem{hv2003} Houdr\'e, C. and Villa, J.: An example of infinite dimensional quasi-helix. \textit{Contemporary Mathematics} {\bf 336}, 195-202, 2003.
		
		\bibitem{km2020}Kalbasi, K. and Mountford, T.: On the probability distribution of the local times of diagonally operator-self-similar 
		Gaussian fields with stationary increments. \textit{Bernoulli} {\bf 26}(2), 1504--1534, 2020.
		\bibitem{no2007}Nualart, D. and Ortiz-Latorre, S.: Intersection local time for two independent fractional Brownian motions, \textit{Journal of Theoretical Probability} {\bf 20}(4), 759--767, 2007.	
		\bibitem{s2013} Sghir, A.: Limit theorems and strong approximation for occupation times problem of sub-fractional Brownian motion in a class of Besov spaces. \textit{Random Operators and Stochastic Equations} {\bf 21}, 111--124, 2013.
		\bibitem{s2011} Shen, G.: Necessary and sufficient condition for the smoothness of intersection local time of subfractional Brownian motions. \textit{Journal of Inequalities and Applications} {\bf 2011}, 1--16, 2011.
		\bibitem{skm1993} Samko, S., Kilbas, A. and Marichev, O.: \textit{Fractional integrals and derivatives: theory and applications}, Gordon and Breach, Switzerland, 1993.
		\bibitem{sxy2019} Song, J., Xu, F. and Yu, Q.: Limit theorems for functionals of two independent Gaussian processes. \textit{Stochastic Processes and their Applications} {\bf 129}(11), 4791--4836, 2019.	
		\bibitem{wx2010} Wu, D. and Xiao, Y.: Regularity of intersection local times of fractional Brownian motions.  \textit{Journal of Theoretical Probability} {\bf 23}(4), 972--1001, 2010.
		\bibitem{xy2023} Xu, X. and Yu, X.: The fractional smoothness of integral functionals driven by Brownian motion. \textit{Statistics \& Probability Letters} {\bf 193}, 109717, 2023.
		\bibitem{y2014} Yan, L.: Derivative for the intersection local time of fractional Brownian motions. arXiv:1403.4102., 2014.		
		\bibitem{y2016}Yan, L.: The fractional derivative for fractional Brownian local time with Hurst index large than 1/2. \textit{Mathematische Zeitschrift} {\bf 283}, 437--468, 2016.	
		\bibitem{yyc2017}Yan, L., Yu, X. and Chen, R.: Derivative of intersection local time of independent symmetric stable motions. \textit{Statistics \& Probability Letters} {\bf 121}, 18--28, 2017.
		\bibitem{y1985}Yamada, T.: On the fractional derivative of Brownian local times. \textit{Jounal of Mathematics of Kyoto University} {\bf 25}(1), 49--58, 1985.
		\bibitem{y1986}Yamada, T.: On some limit theorems for occupation times of one-dimensional Brownian motion and its continuous additive functionals locally of zero energy. \textit{Jounal of Mathematics of Kyoto University} {\bf 26}(2), 302--322, 1986.
		\bibitem{ysz2015}Yin, X., Shen, G. and Zhu, D.: Intersection local time of subfractional Ornstein-Uhlenbeck processes. \textit{Hacettepe Journal of Mathematics and Statistics} {\bf 44}(4), 975--990, 2015.
		\bibitem{zsy2023}Zhou, H., Shen, G. and Yu, Q.: Derivatives of Intersection Local Time for Two Independent Symmetric $\alpha-$stable Processes. \textit{Acta Mathematica Sinica, English Series}, 1--20, 2023.
		\bibitem{l2026} Liu, H.: Functional limit theorems for some self-similar Gaussian processes in critical and subcritical cases. \textit{Statistics \& Probability Letters} {\bf 227}, 110547, 2026.
		\bibitem{djwz2024} Duan, Y., Jiang, Y., Wei, Y. and Zhou, J.: The solution of stochastic evolution equation with the fractional derivative. \textit{Physica Scripta} {\bf 99}(2), 025219, 2024.
		\bibitem{djwx2024} Duan, Y., Jiang, Y., Wei, Y. and Xu, M.: Stochastic wave equation with Marchaud fractional derivative. \textit{Annals of Mathematical Sciences and Applications} {\bf 9}(3), 531--557, 2024.
	\end{thebibliography}
\end{document}